\documentclass[12pt]{amsart}
\usepackage[utf8]{inputenc}
\usepackage{amsmath}
\usepackage{amssymb}
\usepackage{amsthm}
\usepackage{xspace}
\usepackage{breqn}
\usepackage{mathrsfs}
\usepackage{tikz-cd}
\usepackage{enumitem}
\usepackage{blindtext}
\usepackage{numprint}
\usepackage{seqsplit}
\usepackage[all]{xy}
\usepackage{hyperref}

\newcommand{\Z}{{\mathbb Z}}
\newcommand{\Q}{{\mathbb Q}}
\newcommand{\PP}{{\mathbb P}}

\usepackage{setspace}
\usepackage{graphicx}
\graphicspath{ {./} }
\usepackage[OT2,T1]{fontenc}

\DeclareSymbolFont{cyrletters}{OT2}{wncyr}{m}{n}
\DeclareMathSymbol{\Sha}{\mathalpha}{cyrletters}{"58}
\usepackage[margin=1in]{geometry}

\newtheorem{theorem}{Theorem}[section]

\newtheorem{proposition}[theorem]{Proposition}
\newtheorem{lemma}[theorem]{Lemma}

\theoremstyle{definition}

\newtheorem{definition}[theorem]{Definition}
\newtheorem{remark}[theorem]{Remark}
\newtheorem{algorithm}[theorem]{Algorithm}
\newtheorem{example}[theorem]{Example}

\newcommand{\gp}{\ensuremath{\mathfrak{p}{}}\xspace}
\providecommand{\mb}[1]{\mathbb{ #1} }
\providecommand{\mc}[1]{\mathcal{ #1} }

 \providecommand{\floor}[1]{\left \lfloor #1 \right \rfloor }

\renewcommand{\phi}{\varphi}

\begin{document}
\date{29th November, 2021}
\title[Kolyvagin classes]{Explicit realization of elements of the Tate-Shafarevich group constructed from Kolyvagin classes}

\author{Lazar~Radi\v{c}evi\'{c}}
\address{Max Planck, Institute for Mathematics, Vivatsgasse 7,
53111 Bonn,
Germany}
\email{lazaradicevic@gmail.com}

\begin{abstract}
    We consider the Kolyvagin cohomology classes associated to an elliptic curve $E$ defined over $\Q$ from a computational point of view. We explain how to go from a model of a class as an element of $(E(L)/pE(L))^{\mathrm{Gal}(L/\Q)}$, where $p$ is prime and $L$ is a dihedral extension of $\Q$ of degree $2p$, to a geometric model as a genus one curve embedded in $\mb{P}^{p-1}$. We adapt the existing methods to compute Heegner points to our situation, and explicitly compute them as elements of $E(L)$. Finally, we compute explicit equations for several genus one curves that represent non-trivial elements of $\Sha(E/\Q)[p]$, for $p \leq 11$, and hence are counterexamples to the Hasse principle.

\end{abstract}
\maketitle

    \section{Introduction}
	 Let $E/\mathbb{Q}$ be an elliptic curve of conductor $N$, with a fixed modular parametrization $\phi : X_{0}(N) \xrightarrow{} E$. Let $K$ be an imaginary quadratic field satisfying the Heegner hypothesis: all prime factors of $N$ split in $K$. Let $H$ be the Hilbert class field of $K$. Using the theory of complex multiplication and the modular parametrization $\phi$, one defines certain points in $E(H)$, known as Heegner points.

	Let us fix an odd prime $p$. Kolyvagin (\cite{kolyvagin1989finiteness}) has used Heegner points to construct certain elements of the $p$-Selmer group $\mathrm{Sel}^{p}(E/\mb{Q})$. The images of these elements in the Tate-Shafarevich group $\Sha(E/\Q)[p]$, under the natural map $\mathrm{Sel}^{p}(E/\mb{Q}) \xrightarrow{} \Sha(E/\mb{Q})[p]$, are known as Kolyvagin classes. It is a standard fact, due to Cassels (\cite{cassels1962arithmetic}), that elements of $\Sha(E/\Q)[p]$ can be represented by genus one curves embedded in $\mb{P}^{p-1}$.
	
	The main result of this paper is an algorithm (divided into Algorithm \ref{mini algortihm} and Algorithm \ref{heegner point algorithm}) to compute such representations for the Kolyvagin classes in $\Sha(E/\Q)[p]$, and thus obtain explicit counterexamples to the Hasse principle. In Section \ref{koly example section} we then use these algorithms to compute explicit equations for smooth genus curves embedded in $\mb{P}^{p-1}$, that have points over every completion of $\Q$, but no points defined over $\Q$ for primes $p \leq 11$.

		These calculations are especially interesting if $p>5$, and the curve $E$ does not admit a $p$-isogeny. The standard method to compute such counter examples to the Hasse principle is to use the  method of complete $p$-descent to compute the entire $p$-Selmer group $\mathrm{Sel}^{(p)}(E/\Q)$. However, when $p>5$, this is not feasible in practice, as one runs into difficulties with computing class groups of very large number fields. Our method does not run into this problem, and in particular, in Section \ref{koly example section}, Example \ref{7 example},  we  compute the first known explicit realization of a non-trivial element of $\Sha(E/\mathbb{Q})[7]$ for an elliptic curve $E$ that does not have a $7$-isogeny. 
		
		These classes have already been studied from a computational point of 
		view by Jetchev, Leuter and Stein in \cite{JLS}. They are able to compute representations of these classes as  elements of $E(L)/pE(L)$, where $L$ is a certain abelian extension of $K$. However, their aim is only to test whether these classes are non-zero, for which this representation is sufficient, whereas we compute explicit equations defining the corresponding homogeneous space.

		 The problem of computing these equations breaks up into two problems. First, given a suitable elliptic curve $E$, a discriminant $D$ and a prime $p$, we compute a Heegner point $x_K$, defined over a certain dihedral extension of $\mb{Q}$. Our method for doing this is Algorithm \ref{heegner point algorithm}. To this point we associate a Kolyvagin class $c \in \mathrm{Sel}^{(p)}(E/\mb{Q})$. Algorithm \ref{mini algortihm} then represents this class by a genus one curve $C \subset \mb{P}^{p-1}$. The main difficulty in our computations is caused by the fact that typically the Heegner points $x_K$ have very large height, making them hard to compute and work with. We note that despite this, the output of Algorithm \ref{mini algortihm} is a model for the curve $C$ with small integral coefficients, i.e. a \textit{minimized} and \textit{reduced} model, in the sense of \cite{minred234}.
	
			\textbf{Acknowledgements.} The author thanks his PhD advisor Tom Fisher, for suggesting the problem and for his patient guidance along the way. Moreover the author thanks his thesis examiners Jack Thorne and Vladimir Dokchitser for their helpful comments.  This work is based on Chapter 7 of the author's PhD thesis \cite{LazarThesis}. The author is grateful to Trinity College and Max Planck Institute for Mathematics for their financial support.
			
		\section{Background on Kolyvagin classes and statement of results} \label{koly background}
		In this section we review basic material from the theory of Heegner points. The main references are the articles of Gross, \cite{gross} and \cite{gross1984heegner}, as well as \cite{weston2015euler}, \cite{watkins2005some} and Chapter 8 of \cite{cohen2008number}.

		\subsection{Heegner points on modular curves.}
		For $N \geq 1$ an integer, let $Y_0(N)$ be the open modular curve, defined over $\mb{Q}$. The $\mb{C}$-points of $Y_0(N)$ classify isomorphism classes of cyclic $N$-isogenies $E \xrightarrow{} E'$, defined over $\mb{C}$. Fix an imaginary quadratic field $K$ satisfying the Heegner hypothesis: every prime dividing $N$ splits completely in $K$. It follows that there exists an ideal $\mc{N}$ of  the ring of integers $\mc{O}_K$ with $\mc{N}\bar{\mc{N}}=N \mc{O}_K$, and hence $\mc{O}_{K}/ \mc{N} \cong \Z / N \Z$.
		
        Given such an ideal $\mc{N}$, an ideal class $[\mathfrak{a}] \in \mathrm{Cl}(\mc{O}_K)$ determines a map of complex torii $\mb{C}/\mathfrak{a} \xrightarrow{} \mb{C}/\mathfrak{a}\mc{N}^{-1}$. Since we have $\mathfrak{a}\mc{N}^{-1}/\mathfrak{a} \cong \mb{Z}/N\mb{Z}$, this map is a cyclic $N$-isogeny, and determines a point in $Y_0(N)(\mb{C})$.  This is defined to be the Heegner point associated to the triple $(\mc{O}_K,[\mathfrak{a}],\mc{N})$.

		\subsection{Rationality of Heegner points.} \label{cm background} A key property of Heegner points, implied by the theory of complex multiplication, is that they are defined over abelian extensions of the field $K$. More precisely, let $(\mc{O}_K,[\mathfrak{a}],\mc{N})$ be a Heegner point on $Y_0(N)$. This point is defined over the Hilbert class field $H$ of $K$. See $\S5$ of \cite{gross1984heegner}.  The key point is that both $\mb{C}/\mathfrak{a}$ and $\mb{C}/\mathfrak{a}\mc{N}^{-1}$ have complex multiplication by $\mc{O}_K$.  This is a consequence of the Shimura reciprocity law, as explained in Chapter 6.8  of \cite{shimura1971introduction}, or Chapter II of \cite{Sil2}. 
		
		The  field $H$ is an abelian extension of $K$, and the Artin map provides a canonical isomorphism $F:\mathrm{Cl}(\mc{O}_K) \xrightarrow{} \mathrm{Gal}(H/K)$. Explicitly, by Shimura reciprocity, for an ideal class $[\mathfrak{b}]$ we have 
		\[
		(\mc{O}_K,[\mathfrak{a}],\mc{N})^{F([\mathfrak{b}])}=	(\mc{O}_K,[\mathfrak{ab^{-1}}],\mc{N}).
		\]
	Suppose that $\tau\in \mathrm{Gal}(H/\mb{Q})$ is a lift of complex conjugation. The action of $\tau$ is given by  
		\[
		(\mc{O}_K,[\mathfrak{a}],\mc{N})^{\tau}=(\mc{O}_K,[\tau(\mathfrak{a})],\tau(\mc{N})).
		\]
		
		\subsection{Heegner points on elliptic curves and Kolyvagin classes.} Now let $E$ be an elliptic curve defined over $\mb{Q}$, of conductor $N$. Let $X_0(N)$ be the compactified modular curve of level $N$. By the modularity theorem (see \cite{breuil2001modularity}), there exists a modular parametrization map $\phi : X_0(N) \xrightarrow{} E$. For every discriminant $D$ that satisfies the Heegner condition, we fix an ideal $\mc{N}$ with $N \mc{O}_K=\mc{N} \bar{\mc{N}}$, and define the Heegner point $x_{K} \in E(H)$ by setting $x_K=\phi(\mc{O},[\mathfrak{a}],\mc{N})$. We also define the basic Heegner point $y_K \in E(K)$ by setting $y_K=\mathrm{Tr}_{H/K} x_K$. 	
		
		Let $p>2$ be a prime such that $E(H)[p]$ is trivial, $y_K \in pE(K)$ and $p$ divides $|\mathrm{Cl}(\mc{O}_K)|=|H:K|$ exactly once. These assumptions are fairly mild, as we will see later. Then there exists a unique degree $p$ subfield of $H$, which we denote by $L$. Let $z_{K}=\mathrm{Tr}_{H/L} x_{K}$, and let $\sigma$ be a generator of $G=\mathrm{Gal}(L/K).$	Define the operators $D_{\sigma}$ and  $\mathrm{Tr}$ in $\mathbb{Z}[G]$ by  
		\begin{equation*}
			D_{\sigma}=\sum_{i=1}^{p-1} i \sigma^i, \quad  \mathrm{Tr}=\sum^{p-1}_{i=0} \sigma^i.
		\end{equation*}
		The operator $D_{\sigma}$ is known as the \textit{Kolyvagin derivative} and $\mathrm{Tr}$ is the trace operator. They satisfy the identity
		\begin{equation*}
			(\sigma-1)D_{\sigma}=p-\mathrm{Tr},
		\end{equation*}
		We define the \textit{derived} Heegner point $P$ as $P=D_{\sigma} 
		\cdot z_K$. The class $[P] \in E(L)/pE(L)$ is invariant under the 
		action of $G$, since we have
		\[
		(\sigma-1)(D_{\sigma} \cdot z_K)=p z_K- \mathrm{Tr} (z_K)=p z_K-\mathrm{Tr}_{H/K} (x_K)=p z_K-y_K,
		\]
		and by assumption $y_K \in pE(K)$. The Kummer map $\delta : E(L)/p(L) \xrightarrow{} H^1(L,E[p])$is compatible with the Galois action, and so we can define a cohomology class $c_L \in H^1(L,E[p])^{\mathrm{Gal}(L/K)}$ by $c_L=\delta([P])$. We have the inflation-restriction exact sequence
		\begin{equation*}
			H^1(L/K, E[p](L)) \xrightarrow{\mathrm{inf}} H^1(K,E[p]) \xrightarrow{\mathrm{res}} H^1(L,E[p])^{\mathrm{Gal}(L/K)} \xrightarrow{} H^2(L/K, E[p](L)).
		\end{equation*}
		As $E[p](L)$ is trivial, the two outermost groups are trivial, and the restriction map defines an isomorphism $\mathrm{res}: H^1(K,E[p]) \xrightarrow{} H^1(L,E[p])^{\mathrm{Gal}(L/K)} $. We define $c \in H^1(K,E[p])$ to be the preimage of $c_L$ under the restriction map. The class $c$ is in fact an element of the $p$-Selmer group $\mathrm{Sel}^{(p)}(E/K)$, see Prop. 6.2 of \cite{gross}. Finally, let $d$ be the image of $c$ in $H^1(K,E)$. Then $d$ is an element of $\Sha(E/K)[p]$.

	\subsection{Descent from $K$ to $\mb{Q}$.}
	Let $\epsilon$ be the sign of the functional equation of $E/\mathbb{Q}$. The proof of Proposition 5.4 in \cite{gross} shows that the class $c$ lies in the $\epsilon$-eigenspace for the action of complex conjugation on $H^1(K,E[p])$. Thus, if $E$ is a curve of rank $0$, $c$ is fixed by complex conjugation, and by the same inflation-restriction argument we naturally obtain an element of $H^1(\mathbb{Q},E[p])$, which we will also call $c$.
	
	As $E$ has no non-trivial $p$-torsion and rank 0, the group $E(\mathbb{Q})/pE(\mathbb{Q})$ is trivial, and hence if $c$ is non-zero, its image $d$ in $H^1(K,E[p])$ will be a non-trivial element of $\Sha(E/\mathbb{Q})[p]$. Tracing through the isomorphisms used to define $c$, we see that the class $c$ is non-zero if and only if the point $P$ is not divisible by $p$ in $E(L)$.

		\subsection{Galois cohomology and $n$-diagrams.}
		Let $F$ be a number field, $E/F$ an elliptic curve and $n \geq 1$ an integer. We briefly recall a few standard facts about the Galois cohomology groups $H^1(F,E)$ and $H^1(F,E[n])$, see for example \cite{descentpaper}. 
		
	A torsor under $E$ is a smooth projective curve $T/F$, together with a regular simply transitive action of $E$ on $T$.	An isomorphism of torsors $C_1$ and $C_2$ is an isomorphism of curves $C_1$ and $C_2$ that respects the action of $E$. The left action of $E$ on itself by translations makes $E$ a torsor, which we call the trivial torsor. There is a natural identification of the group $H^1(F,E)$ with the set of isomorphism classes of torsors defined over $F$, and the trivial torsor $E$ corresponds to the identity element.
		
		We will also need the following interpretation of the group $H^1(F,E[n])$. We define a diagram $[C \xrightarrow[]{} S]$ to be a morphism from a torsor $C$ to a variety $S$. An isomorphism of diagrams $[C_1 \xrightarrow{} S_1] \sim [C_2 \xrightarrow{} S_2]$ is an isomorphism
		of torsors $\phi : C_1 \cong C_2$ together with an isomorphism of varieties $\psi : S_1 \cong S_2$ making the diagram
		
		\[ \begin{tikzcd}
			C_1 \arrow{r}{} \arrow[swap]{d}{\phi} & S_1 \arrow{d}{\psi} \\%
			C_2 \arrow{r}{}& S_2
		\end{tikzcd}
		\]
		commute. We define the trivial $n$-diagram to be the diagram $[E \xrightarrow[]{} \mb{P}^{n-1}]$ where the morphism is induced by the complete linear system of the divisor $n \cdot 0_E$, and in general, we say a diagram $[C \xrightarrow[]{} S]$ is an $n$-diagram if it is defined over $F$, but isomorphic to the trivial diagram over the algebraic closure $\bar{F}$, i.e. a twist of the trivial diagram. The set of isomorphism classes is also naturally identified with $H^1(F,E[n])$.
		
        The group law on $E$ induces a summation map $\mathrm{sum} : \mathrm{Div} E \xrightarrow[]{} E$, given by $\sum n_p \cdot (P) \mapsto \sum n_pP.$ Two divisors $D$ and $D'$ of the same degree are linearly equivalent if and only $\mathrm{sum}(D)=\mathrm{sum}(D')$. The Kummer map $\delta : E(F)/nE(F) \xrightarrow{}H^1(F,E[n])$ sends a class $[P] \in E(F)/nE(F)$ to the isomorphism class of the $n$-diagram $[E \xrightarrow[]{} \mb{P}^{n-1}]$, where the map is induced by the complete linear system of any degree $n$ divisor $D$ with $\mathrm{sum}(D)=P$.
	
		In this article we consider only $n$-diagrams of the form $[C \xrightarrow[]{} \PP^{n-1}]$. When $n \geq 3$, such an $n$-diagram is a closed embedding, and its image is a smooth projectively normal curve $C$ of genus one and degree $n$. If $n=3$, $C$ is a plane cubic. For $n \geq 4$, the ideal defining $C$ is generated by $n(n-3)/2$ quadrics, see \cite[Prop.~5.3]{g1inv}. Finally, as a consequence of class field theory, under the above identification, the elements of the $n$-Selmer group of $E$ can be represented by $n$-diagrams of the form $[C \xrightarrow[]{} \mb{P}^{n-1}]$.
		
		\subsection{Summary of the setup and the statement of results.} \label{summary subsection}
		Our starting data is an elliptic curve $E/\Q$ of rank $0$ and an odd prime $p$ for which the Birch and Swinnerton-Dyer conjecture predicts that the group $\Sha(E/\Q)[p]$ is non-trivial. To construct a Kolyvagin class, we also need to find a discriminant $D$ of an imaginary quadratic field $K$ with Hilbert class field $H$ that satisfies the following: $D$ satisfies the Heegner hypothesis, $E(H)[p]$ is trivial, $p$ divides $|\mathrm{Cl}(\mc{O}_K)|=|H:K|$ exactly once, the rank of $E/K$ is 0, and the basic Heegner point $y_K$ is divisible by $p$ in $E(K)$. 
		
		\begin{remark}
		For given $E$ and $p$, it is usually easy to find a discriminant $D$ that satisfies these conditions. In practice, a naive search will usually find plenty of discriminants that are easily seen to satisfy all conditions but the last one, and a famous theorem of Kolyvagin (Theorem 1.3 of \cite{gross}) then often guarantees that we must have $p|y_K$. 
		\end{remark}
		
		Starting from the data of $E,p$ and $D$, we compute a $p$-diagram representing the Kolyvagin class $c \in \mathrm{Sel}^{(p)}(E/\Q)$ defined above. There are two main steps. Algorithm \ref{heegner point algorithm} computes the Heegner point $x_K$ as an element of $E(H)$, and using this data Algorithm \ref{mini algortihm} then computes the equations defining a genus one normal curve $C \subset \mb{P}^{p-1}$, and the inclusion $C \subset \mb{P}^{p-1}$ is the $p$-diagram representing the class $c$. If this class is non-trival, the curve $C$ is a counter-example to the Hasse principle. We are able to succesfully use these algorithms for various elliptic curves with $p=3,5$ and $7$, and we give examples in Section \ref{koly example section}. Note that the examples with $p=3$ and $p=5$ can also be obtained by the method of $p$-descent, but Example \ref{7 example} with $p=7$ is out of reach of $p$-descent at the moment, and is the first such example with no 7-isogeny to our knowledge.

			\begin{remark}
		If the curve $E$ has rank 1 over $\Q$, then the class $d$ is in the $(-1)$-eigenspace of complex conjugation, and hence is obtained as the restriction of an element of $\Sha(E_{D}/\mb{Q})[p]$, where $E_D$ is the quadratic twist of $E$ by $D$. If this quadratic twist has rank 0, then by the same argument, the class $d$ is non-zero if and only if the class $c$ is. 
		
		Our method applies in this case as well, and in fact we are able to compute an example (Example \ref{11 example}) with $p=11$, i.e. a genus one normal curve $C  \subset \mb{P}^{10}$ that is a counter-example to the Hasse principle. We suspect that $p$ could be increased even further, since computing the Heegner point appears to be much easier in this case. As the Kolyvagin class is naturally an element of $\Sha(E_D/\Q)$ for the quadratic twist $E_D$ of the curve $E$ we begin with, it seems difficult to use this version of our method as a tool to compute $\Sha(E/\mb{Q})[p]$ of a given curve $E$.
		\end{remark}

		\section{Computing the Heegner point}

		In this section we describe the algorithm we will use to compute the Heegner points needed to define the Kolyvagin class.
		
	    \subsection{Computing the modular parametrization}
		We briefly recall how to compute a modular parametrization of an elliptic curve, following \cite{cohen2008number} and \cite{JLS}. Let $E/\mathbb{Q}$ be an elliptic curve of conductor $N$, $p$ an odd 
		prime, and let $K$ be an imaginary quadratic 
		field. We assume that the maximal order $\mc{O}_K$ of $K$ is of  discriminant $-D\neq3,4$ and that $K$ satisfies the Heegner hypothesis: 
		all prime factors of $N$ split completely in $\mc{O}_K$. We fix an ideal $\mc{N}$ with $N \mc{O}_K= \mc{N}\bar{\mc{N}}$. Let $H$ be the Hilbert class field of $K$, and fix a modular parametrization $\phi : X_0(N) \xrightarrow{} E$ that 
		maps the cusp $\infty$ to the origin of $E$. We assume that $H$ has unique subfield of degree $p$ over $K$, denoted $L$. The Heegner point 
		$x_K$ is defined to be the image 
		$x_K=\phi(\mc{O}_K,[\mc{O}_K],\mc{N}) \in E(H)$, and we set $z_{K}=\mathrm{Tr}_{H/L} x_{K}$. 

		Following \cite{JLS}, we give an explicit description of the map $\phi$. Let $\Lambda$ be the period lattice associated to $E$, and  let $f\in S_2(\Gamma_0(N))$ be the newform associated to $E$. Let $\mc{H}^{*}=\mc{H}\cup \mathbb{Q} \cup \{\infty\}$ be the extended upper half plane, and identify the modular curve $X_0(N)$ with the quotient $\mc{H}^*/\Gamma_0(N)$. The modular parametrization map $\phi:X_0(N) \xrightarrow{} \mathbb{C}/\Lambda$ is given by integrating the holomorphic differential $f(z)dz$ on $X_0(N)$. We can compute it using the power series
		\begin{equation} \label{modular parametrization power series eq}
		\phi(\tau)=\int^{\infty}_{\tau} f(z)dz=\sum_{n \geq 1} \frac{a_n}{n} e^{2\pi i n \tau},
		\end{equation}
		where $f=\sum_{n\geq 1}a_n q^n$ is the Fourier expansion of $f$. To obtain a parametrization $X_0(N) \xrightarrow{} E$, we compose with the uniformization  $\psi=\mathbb{C}/\Lambda \xrightarrow{} E(\mb{C})$. The map $\psi$ is defined using the Weierstrass $\wp$-function, and is is easy to compute numerically to high precision. 
		
		The Artin map provides us with an isomorphism between the class group $\mathrm{Cl}(\mc{O}_K)$ and the Galois group $\mathrm{Gal}(L/K)$. We first compute a set of representatives $\mathfrak{a}_1,...,\mathfrak{a}_m$ for $\mathrm{Cl}_{K}$. Let $\sigma_i$ be the image of $\mathfrak{a}_i$ under the Artin map.  The Galois conjugates of the point $x$ corresponding to the isogeny $[\mathbb{C}/\mathcal{O}_{K} \xrightarrow{} \mathbb{C}/\mathcal{N}^{-1}]$ are given by
		\begin{equation} \label{galois action on hp}
		\sigma_i(x)= [\mathbb{C}/\mathfrak{a}^{-1}_i \xrightarrow{} \mathbb{C}/\mathfrak{a}^{-1}_i \mathcal{N}^{-1}].
		\end{equation}
		
		For every conjugate we compute a corresponding point $\tau$ in the 
		upper half plane.  Fix an embedding $i$
		of $L$ into $\mathbb{C}$. As the morphism $\phi$ is defined over 
		$\mathbb{Q}$,  we can use the above description of the Galois action to compute the coordinates of 
		the Galois conjugates $\sigma_i(x_K)$,  to 
		whatever precision we like, and then do the same for the point $z_K=\mathrm{Tr}_{H/L}x_K$. 
		
		\subsection{Recognizing the Heegner point using lattice reduction}
		
	    We now discuss how to use the LLL algorithm to recover the point $z_K$ from the data computed in the previous section. Recognizing an algebraic number from floating point approximations is a well-studied problem, and in the setting of Heegner points has been considered in \cite{JLS}. An algorithm 
		 similar to the one they propose has been implemented in MAGMA by Steve 
		 Donnelly. For our purposes however, their method is too slow to handle 
		 the case when $p \geq 5$,  so in this section we propose a variant to 
		 this method that seems to work quite well in this setting.
	    
	    Let $L$ be a number field of degree $n$, with $n$ complex embeddings $\sigma_1,\ldots,\sigma_{n}$, and let $\alpha_1,...,\alpha_{n}$ be a $\Z$-basis of the ring integers $\mathcal{O}_{L}$. In our application $L$ will be a dihedral extension of $\Q$ of degree $2p$.
		
		\begin{definition} \label{recognition matrix def}
			Let $\epsilon >0$, $C=10^B$ for an integer $B>0$, $z \in L$  and let $z_1,...,z_{n} \in \mathbb{C}$ be such that $|\sigma_i(z) -z_i|<\epsilon$. Let $\alpha_{ij} \in \mathbb{C}$ be such that $|\sigma_j(\alpha_i)-\alpha_{ij}|<\epsilon$. To this data we associate the $2n\times3n$ integer matrix $A_{z,\epsilon,C}$:
		
				\[
			\begin{pmatrix}
			1 & 0  & \hdots & 0 & 0 &  \floor{C\alpha_{1,1}}  &\hdots & \floor{C\alpha_{1,n}}\\
			\vdots & \vdots & \ddots  & \vdots & \vdots & \vdots &\hdots &  \vdots\\
			\vdots & \vdots & \hdots &\vdots & \vdots &  \floor{C\alpha_{n,1}} &\hdots &\floor{C\alpha_{n,n}} \\
			\vdots & \vdots & \hdots &\vdots & \vdots &  \floor{C\alpha_{1,1} z_1} &\hdots &\floor{C\alpha_{1,n} z_{n}} \\
			
			\vdots & \vdots  & \ddots & \vdots & \vdots & \vdots & \vdots &\vdots\\
			0 & 0 & \hdots & 0 & 1 &\floor{C \alpha_{n,1} z_1}& \hdots&\floor{C\alpha_{n,n} z_{n}}
			\end{pmatrix}
			\]
			i.e. the left $2n\times 2n$-block is the $2n\times 2n$ identity matrix, and the right $2n\times n$-second block splits into the upper $n \times n$ block $(\floor{C\alpha_{ij}})_{ij}$ and the lower $n \times n$ block $(\floor{C\alpha_{ij} z_j})_{ij}$. We define $L_{z,\epsilon,C}$ to be the lattice in $\mathbb{R}^{3n}$ spanned by the rows of $A_{z,\epsilon,C}$.
		\end{definition}
		
		To recover $z$ from  $A_{z,\epsilon,C}$, our strategy is to use the LLL lattice reduction algorithm to find short vectors in the lattice $L_{z,\epsilon,C}$. We take $C=10^B$ to be a large constant and $\epsilon$ as small as possible. Let the rows of $A_{z,\epsilon,C}$ be $r_1,r_2,...,r_{2n}$. The lattice reduction algorithm gives us integers $u_1,u_2,...,u_{2n}$ such that the row vector $\sum^{2n}_{i=1} u_i r_i$ is "small", and if $\epsilon$ is small enough, we hope that we have 
		\begin{equation*}
		0=u_1\alpha_1+u_2\alpha_2+...+u_n \alpha_{n}+u_{n+1}\alpha_{2}z+u_{n+2}\alpha_{2}z+...+u_{2n}\alpha_{n}z,
		\end{equation*}
		If this is the case, we can recover $z$ from the expression 
		\begin{equation*} 
		z=-\frac{\sum^{n}_{i=1} u_i \alpha_i}{\sum^{n}_{i=1} u_{n+i} \alpha_i}.
		\end{equation*}
		
		We summarise the discussion of this section in the following algorithm.
		
		\begin{algorithm} \label{heegner point algorithm}
			\begin{itemize}
				\item[]
				\item INPUT: An elliptic curve $E$, a Heegner discriminant $D$, and a prime $p$ that divides $|\mathrm{Cl}(\mc{O}_K)|$ exactly once.
				\item OUTPUT: Coordinates $(x,y)$ of a point $P \in E(L)$ that is (conjecturally) the point $z_K=\mathrm{Tr}_{H/L}x_K$.  
			\end{itemize}  
			\begin{enumerate}
				\item Find a set of representatives $\mathfrak{a}_1,\mathfrak{a}_2,...,\mathfrak{a}_{n}$ for the class group $\mathrm{Cl}(K)$, and for each point $[\mathbb{C}/\mathfrak{a}^{-1}_i \xrightarrow{} \mathbb{C}/\mathfrak{a}^{-1}_i \mathcal{N}^{-1}]$, compute a corresponding $\tau_i$ in the upper half plane.
				\item Compute an equation for the Hilbert class field $L$, and compute an integral basis of the maximal order $\mathcal{O}_{L}$.
				\item Pick an $\epsilon>0$, and compute $\phi(\tau_i) \in \mathbb{C}$ to precision $\epsilon/2$, using the formula (\ref{modular parametrization power series eq}), by computing enough of the Fourier coefficients $a_n$.
				\item Compute the period lattice $\Lambda$ and hence the uniformisation map $\psi :\mathbb{C}/\Lambda \xrightarrow[]{} E$ to the required precision, and hence find $\psi(\phi(\tau_i))$. Then, use the description of the Galois action on Heegner points given in Section \ref{cm background} to take the trace from $H$ to $L$, and hence obtain $z_1,...,z_{2p}$ with $|\sigma_i(x)-z_i|<\epsilon$.
				\item Using $z_1,z_2,...,z_{2p}$ and choosing a large constant $C$, form the matrix $A_{z,\epsilon,C}$ as in Definition \ref{recognition matrix def}. Use the LLL algorithm to find a $U\in \mathrm{SL}_{4p}(\mathbb{Z})$ such that the rows of $UA_{z,\epsilon,C}$ form an LLL-reduced basis of $L_{z,\epsilon,C}$. Then let  $x=-\frac{\sum^{2p}_{i=1} u_{1,i} \alpha_i}{\sum^{2p}_{i=1} u_{1,2p+i} \alpha_i}$ and test if $x$ is the $x$-coordinate of a point in $E(L)$. If it is, solve for the $y$-coordinate and return $(x,y)$. Otherwise, replace $\epsilon$ by $\epsilon/2$, and return to Step 3.
			\end{enumerate}
		\end{algorithm}
		
		Steps (i), (ii) and (iii) of the algorithm have been studied extensively in the literature, see for example Section 8.6 of \cite{cohen2008number} or \cite{watkins2005some}, so we do not provide details on how to implement them. We have used the existing MAGMA implementations of these steps in our calculations.
		The algorithm has not proven that the point $(x,y)$ is indeed the point $x_K$, although we believe it is highly probable that it is, nor have we proven that the algorithm always terminates. However, in practice we have been able to use it to compute points on various for $p\leq 11$. 
		
		The main bottleneck is Step (iii). If the height of the Heegner point is very large, then we need to take $\epsilon$ to be very small for the algorithm to return a point in $E(L)$, and this requires computing a very large number of the Fourier coefficients $a_n$.
		
		\begin{remark} \label{remark verified heegner} The output of our algorithm, if it terminates, will be a point $u=(x',y') \in E(H)$, and as noted in \cite{JLS}, verifying that this point coincides with the Heegner point $z_K=(x,y)$ is a nontrivial matter. We know that the point we obtain is a good archimedean approximation of the Heegner point, in the sense that by increasing the precision in Algorithm \ref{heegner point algorithm} we can make the  absolute values $|\sigma(x)-\sigma(x')|$, where $\sigma$ is any embedding $L \hookrightarrow \mb{C}$, as small as we like. However, without a bound on the height of $z_K$, we can't actually prove that $u=z_K$.
		
		Since our main goal is to construct examples of non-trivial elements of the Tate-Shafarevich group of $E$, it suffices to verify that the point $u$ satisfies the same properties as the Heegner point $z_K$ for the purpose of constructing a Kolyvagin class, as formalized in Lemma \ref{cocycle point lemma}.  This also serves as a consistency check on our calculations, and in all of the examples we have computed, we believe it is very unlikely that the resulting point is not the Heegner point. Note that the appendix of \cite{JLS} provides a method one could use to compute a bound on the height of $x_K$ and hence make the calculations provably correct, but we did not implement this algorithm. 
		
		\end{remark}
        \begin{remark} \label{heegner method remark}
        The idea to use the LLL-algorithm method to recover $z$ from the matrix $A_{z,\epsilon,C}$ is very well known. Our approach differs from the standard method explained in Chapter 7 of \cite{hanrot2009lll}. Briefly, the standard method to recover an algebraic number $z$ from a set of complex numbers $\{z_g : g \in G\}$ that approximate the Galois conjugates of $z$ is to approximate the minimal polynomial $f$ of $z$ by $\prod_{g \in G}(x-z_g)$, and try to recognize the coefficients of $f$ as rationals, using continued fractions, or better yet the LLL algorithm. However, for us this method is not efficient enough, since $z$ is of very large height in the examples we consider, and the coefficients of the polynomial $f$ are symmetric polynomials in $z_g$, and hence are of even larger height. 
		
		Our method instead tries to recognize $z$ directly, taking advantage of the fact that we know that $x$ is defined over the Hilbert class field $H$. We  can compute this field beforehand, using the machinery of computational class field theory already implemented in MAGMA.
\end{remark}
		
		\begin{remark}
			A further improvement along the same lines   is to use the fact that we can also compute numerically the $y$-coordinate, and look for linear relations of the form $A+Bx+Cy=0$. Recall that we have assumed that $L$ is of class number 1, so that $x=r/t^2$ and $y=s/t^3$, for some $r,s,t \in \mc{O}_L$. Thus if $A+Bx+Cy=0$, we see that $t|A$, and that hence $A^2x \in \mc{O}_L$.  It is then simple to recover $A^2x$ from its floating point approximation, and hence compute the point $(x,y)$. Based on experimental data we have computed, this seems to be an improvement. A heuristic explanation might be that the minimal $A,B,C$ that can appear in a relation $A+Bx+Cy=0$ can be a lot smaller than the minimal $u,v$ appearing in a relation $u+vx=0$, and so it is easier to guess a short vector in the corresponding lattice.
		\end{remark}

		\section{Geometric realization of the Kolyvagin class} \label{geometric realization section} 
		
		\subsection{The  $p$-diagram associated to the Kolyvagin class} \label{ndiag koly section}
In this section we explain how to compute, given a Heegner point, equations for the $p$-diagram representing the Kolyvagin class. We first formalize the input we need from Heegner points. 

Throughout this section, we fix the following data. Let $E/\mathbb{Q}$ be an elliptic curve of rank 0, let $p$ be an odd prime and $K/\mathbb{Q}$ be a quadratic field. In addition, let $L/\mb{Q}$ be a dihedral extension, of degree $2p$, that contains $K$, such that $E(L)[p]$ is trivial.
		
\begin{proposition}
		Let $P \in E(L)$ be a point such that the class $[P]\in E(L)/pE(L)$ is invariant under the action of $G=\mathrm{Gal}(L/\mb{Q})$. Let  $\delta :E(L)/pE(L) \xrightarrow{} H^1(L,E[p])$ be the Kummer map, and let $\mathrm{res} : H^1(\mb{Q},E[p]) \xrightarrow{} H^1(L,E[p])$ be the restriction map. 

		Then there exists a unique class $c \in H^1(\mathbb{Q},E[p])$ such that $\mathrm{res}(c)=\delta([P])$.
	\end{proposition}
\begin{proof}
	This is the inflation-restriction argument from Section \ref{koly background}.
\end{proof}
The aim of this section is to give method to compute equations for the $p$-diagram representing the class $c$. This is accomplished by Galois descent, and involves explicit cocycle calculations. Let $\sigma \in G$ be an element of order $p$, and let $\tau \in G$ be an involution. 
		\begin{lemma} \label{cocycle point lemma}
		Let $P \in E(L)$ be a point with $[P] \in (E(L)/pE(L))^{G}$, and suppose that we also have $\tau(P)=P$. We then have the following.
			
			\begin{enumerate}[label=(\roman*)]
				\item For each $g \in G$, there exists a unique $R_g \in E(L)$ with $pR_g=g(P)-P$. The map $g \mapsto R_g$ defines a cocycle in $H^1(G,E(L))$, meaning that for any $g,h \in G$ we have
				\begin{equation*} 
				R_{gh}=g(R_h)+R_g,
				\end{equation*}
				
				\item For $0 \leq k \leq p-1$, we have $R_{\sigma^k}=\sum^{k}_{i=1} \sigma^{i-1}(R_{\sigma})$, and $R_{\sigma^k\tau}=R_{\sigma^k}$.
				
				\item We have $\sum^{p}_{k=1} \sigma^k (R_{\sigma})=0_E$ and $\tau(R_{\sigma^k})=R_{\sigma^{p-k}}$.
				
				\item  We have $[P]=[D_{\sigma} R_{\sigma}]=[\sum^{p-1}_{i=1} i\sigma^i (R_{\sigma})]=[\sum^{p-1}_{i=1} R_{\sigma^i}] \in E(L)/pE(L)$.
			\end{enumerate}
		\end{lemma}
		
		\begin{proof}
Since $[g(P)]=[P] \in E(L)/pE(L)$ for every $g \in G$, there exists $R_{g}\in E(L)$ with $pR_g=g(P)-P$. Since $E(L)[p]$ is trivial, $R_g$ is unique, and the cocycle condition follows from
			\[
			pR_{gh}=gh(P)-P=g(h(P)-P)+g(P)-P=pg(R_h)+pR_{g}=p(g(R_h)+R_g),
			\]
			proving (i). Parts (ii) and (iii) then follow from (i) and the identity $\sigma^{p-k}\tau=\tau \sigma^k$.	For (iv), using (ii) we see that 
			\[\sum^{p-1}_{i=1} R_{\sigma^i}=\sum^{p-1}_{i=1}\sum^{i}_{j=1} \sigma^{j-1}(R_{\sigma})=\sum^{p-1}_{i=1} i\sigma^i (R_{\sigma})= D_{\sigma}R_{\sigma}.\]
		By the identity $(\sigma-1) \cdot D_{\sigma}=p-\sum^{p}_{i=1} \sigma^i$ and (iii), we have  $\sigma(D_{\sigma}R_{\sigma})-D_{\sigma}R_{\sigma}=pR_{\sigma}-\sum^{p}_{i=1} \sigma^i( R_{\sigma})=pR_{\sigma}$. Using (iii), $\tau(D_{\sigma}R_{\sigma})=\tau(\sum_{i=1}^{p} R_{\sigma^i})=\sum_{i=1}^{p} R_{\sigma^{p-i}}=D_{\sigma}R_{\sigma}$. Let $Q=P-D_{\sigma}R_{\sigma}$. We have $\sigma(Q)=\tau(Q)=Q$, and hence $Q \in E(\mathbb{Q})$. But we have assumed that $E(\mathbb{Q})$ is finite, so $Q$ is a torsion point. As $E(L)[p]$ is trivial, the image of $Q$ in $E(L)/pE(L)$ is zero, and hence $[P]=[D_{\sigma}R_{\sigma}]$, as desired. 
		\end{proof}		
 We now describe the $p$-diagram corresponding to $c_L$ and the action of the Galois group on this diagram. For a point $Q \in E(L)$, let $\phi_{Q} : E \xrightarrow[]{} E$ be the translation by $Q$ morphism.
		\begin{proposition} \label{galois action divisor prop}
			\begin{enumerate}[label=(\roman*)]
				\item Consider the degree $p$ divisor $D$ on $E$, defined by $D=\sum^{p}_{i=1} R_{\sigma^i}$. Let $l_1,...,l_p$ be a basis of the Riemann-Roch space $\mathcal{L}(D)$ and let $E \xrightarrow{l} \mathbb{P}^{p-1}$ be the embedding induced by this choice of basis. Then $[E \xrightarrow{l} \mathbb{P}^{p-1}]$ is the $p$-diagram representing $c_L \in H^1(L,E[p])$.
				\item The action of the Galois  group $G$ on the divisor $D$ is given by
				\[
				g\left(\sum_{i=0}^{p-1} R_{\sigma^i}\right)=\phi_{R_g}^{*}\left(\sum_{i=0}^{p-1} R_{\sigma^i}\right).
				\]
				
				\item For each $g\in G$, the translation map $\phi_{R_g}$ induces an isomorphism  of $p$-diagrams $g \cdot [E \xrightarrow{ l} \mathbb{P}^{p-1}]$ and $[E \xrightarrow{l} \mathbb{P}^{p-1}]$, represented by the commutative diagram 
				\begin{equation} \label{cocycle diagram} \begin{tikzcd}
				E \arrow{r}{g \cdot l} \arrow[swap]{d}{\phi_{R_{g}}} & \mathbb{P}^{p-1} \arrow{d}{M_{g}} \\%
				E \arrow{r}{l}& \mathbb{P}^{p-1}
				\end{tikzcd}
				\end{equation}
				where $M_g \in \mathrm{PGL}_p(L)$. The map $g \mapsto M_g$ determines a cocycle class in $H^1(G,\mathrm{PGL}_p(L))$.
			\end{enumerate}
		\end{proposition}
		
		\begin{proof}
		 By Lemma \ref{cocycle point lemma}(iv), we have $[\mathrm{sum}(D)]=[D_{\sigma}  R_{\sigma}]=[P] \in E(L)/pE(L)$. Then part (i) follows from the description of the Kummer map given in Section \ref{koly background}. For part (ii), by Lemma \ref{cocycle point lemma}(ii) and (iii), we have $\sigma(R_{\sigma}^{i})=R_{\sigma^{i}}-R_{\sigma}$ and $\tau(R_{\sigma^{i}})=R_{\sigma^{p-i}}$. From these identities we obtain
			\[
			\sigma\left(\sum_{i=0}^{p-1} R_{\sigma^i}\right)=\phi_{R_{\sigma}}^{*}\left(\sum_{i=0}^{p-1} R_{\sigma^i} \right),
		 \hspace{10pt}
			\tau\left(\sum_{i=0}^{p-1} R_{\sigma^i}\right)=\sum_{i=0}^{p-1}  R_{\sigma^i}.
			\]
			The cocycle condition for $g \mapsto R_g$ implies the result for all $g \in G$. 
	
			To see the isomorphism of $p$-diagrams in (iii), note that by (ii) we have $\phi^{*}_{-R_g}(D)=g(D)$, and that $g(l_1),...,g(l_p)$ and $\phi^{*}_{-R_g}(l_1),...\phi^{*}_{-R_g}(l_p)$ are two bases of $\mathcal{L}(g(D))$. We can then take $M_g$ to be the matrix taking one basis to the other. Finally, to see that $g \mapsto M_g$ is a cocycle, let $C_L$ be the image of $E$ in $\mathbb{P}^{p-1}$. $C_L$ is a genus one normal curve of degree $p$, so in particular it spans $\mathbb{P}^{p-1}$, and $M_g$ restricted to $C_L$ is equal to $\phi_{R_g}$. As $g \mapsto \phi_{R_g}$ is a cocycle, we deduce that $M_{gh}=g(M_h)M_g$ holds on $C_L$, and as $C_L$ spans $\mathbb{P}^{p-1}$ and $M_g$ is an automorphism of $\mb{P}^{p-1}$, $M_{gh}=g(M_h)M_g$ must hold on the entire $\mathbb{P}^{p-1}$.
		\end{proof}
	
	The Galois group $G$ acts in a natural way on the field $\mc{L}(E)$ of rational functions on $E$. Explicitly, for $g \in G$ and $f=u/v \in \mc{L}(E)$, with $u,v \in L[x,y]$, we have $g(f)=g(u)/g(v)$, where $g$ acts on $u$ and $v$ by acting on their coefficients. Using this action, we define a twisted action of $G$ on $\mc{L}(E)$ by setting $g \star f=\phi^{*}_{-R_g}(g(f))$. That this is a group action follows immediately from the cocycle condition for $g \mapsto R_g$. By Proposition \ref{galois action divisor prop}(ii), the action restricts to an action on the space $\mc{L}(D)$.
	
	The action $\star$ is semilinear, meaning that we have $g(v+w)=g(v)+g(w)$ and $g(\alpha v)=g(\alpha) g(v)$ for all $v,w \in V$, $\alpha \in L$ and $g \in \mathrm{Gal}(L/\mathbb{Q})$. We need the following standard result, which is equivalent to (generalized) Hilbert's theorem 90.
	
	\begin{lemma} \label{semilinear hilbert 90}
	        Let $V$ be an $n$-dimensional $L$-vector space with a semilinear action of $\mathrm{Gal}(L/\Q)$.
	    	The set of invariant elements $V(D)^{G}$ is an $n$-dimensional $\mathbb{Q}$-vector subspace of $V$. We have $V \cong V^{G} \otimes_{\mathbb{Q}} L$, i.e. $V$ has a basis of $G$-invariant vectors. 
	\end{lemma}
\begin{proof}
     Follows immediately from Lemma 5.8.1 in Chapter II of \cite{silverman2009arithmetic}.
\end{proof}

\begin{remark} \label{linear algebra trace remark}
For $V$ as in the above lemma, the trace map $V \xrightarrow[]{} V^{G}$ is surjective. In other words, if  $\alpha_1,\alpha_2,...,\alpha_{2p}$ is a basis of $L$ over $\mathbb{Q}$, then $V^{G}$ is spanned by the elements $\sum_{g \in G} g(\alpha_i v)$ for $1 \leq i \leq 2p$, $v \in V$. This provides a simple method to compute a basis of $V^G$.
\end{remark}

	Let $l_1,\ldots,l_p$ be a basis of $\mc{L}(D)$, and for each $g\in G$, define $N_g \in  \mathrm{GL}_p(L)$ be the matrix representing the action of $g$ on $\mc{L}(D)$ with respect to this basis.  Then the matrix $N^{-1}_g$ represents the automorphism $M_g \in \mathrm{PGL}_p(L)$ defined in Proposition \ref{galois action divisor prop}(iii), and slightly abusing notation, we write $M_g=N^{-1}_g$.
	\begin{remark}
	The semilinearity of the action $\star$ immediately implies that $g \mapsto M_g$ is a cocycle taking values in $\mathrm{GL}_p(L)$, i.e. we lifted the cocycle $g \mapsto M_g$ to an element of $Z^1(G,\mathrm{GL}_{p}(L)$. This can also be interpreted as showing that the obstruction (\cite{cremona2009explicit}, \cite{o2002period}) of the class $c \in H^1(\Q,E[p])$ in the Brauer group vanishes.
    \end{remark}
\begin{proposition} \label{Galois descent prop}
    	Let $f_{1},f_2,...,f_{p}$ be a basis of $\mathcal{L}(D)$ invariant under the action $\star$. Then the image $C_{\mathbb{Q}}$ of $E$ under the embedding $X \xrightarrow{} [f_1(X):f_2(X):...:f_p(X)]$ can be defined over $\mathbb{Q}$, i.e. the ideal defining $C_{\mathbb{Q}}$ as a projective curve has a basis consisting of polynomials with rational coefficients. The $p$-diagram $C_{\mathbb{Q}} \xrightarrow{} \mathbb{P}^{p-1}$ represents the Kolyvagin class $c$.
\end{proposition}

\begin{proof}
Since the $f_{i}$ are invariant under the action $\star$, for each $g$ the matrix $N_g$ that represents this action is the identity matrix. For each $g \in G$, let $g(C_\Q)$ be the image of $C_{\Q}$ under the standard action of $G$ on $\PP^{p-1}$, i.e. $g \cdot (u_1 : \ldots : u_p)=(g(u_1) : \ldots: g(u_p))$. By Proposition \ref{galois action divisor prop}(iii), we have $g(C_{\mathbb{Q}})=C_{\mathbb{Q}}$ for all $g \in G$.

Let $I$ be the ideal defining $C_{\mathbb{Q}}$. It is generated by a set of $p(p-3)/2$ quadratic forms if $p\geq 5$, and if $p=3$, it is generated by a ternary cubic form. For $p \geq 5$, the $L$-vector space of quadrics vanishing on $C_{\mathbb{Q}}$ is stable under the natural semilinear action of $G$, and hence, by Lemma \ref{semilinear hilbert 90}, has a basis consisting of $G$-invariant elements, i.e. quadrics with rational coefficients. Similarly, if $p=3$, there exists a rational ternary cubic defining $C_{\mathbb{Q}}$. In any case, $[C_{\Q} \subset \mb{P}^{p-1}]$ is a $p$-diagram defined over $\Q$, which represents a class in $H^1(\Q,E[p])$ that restricts to the class $c_L \in H^1(L,E[p])$, and hence is a $p$-diagram that represents the Kolyvagin class.
\end{proof}

The above proposition thus reduces our problem to computing a basis of $\mc{L}(D)^{G}$. Once we compute the cocycle $M_g$ representing the action $\star$ relative to a basis of $\mc{L}(D)$, this is just linear algebra, see Remark \ref{linear algebra trace remark}. 
	\subsection{Computing the matrices $M_{g}$}\label{formula subsection}

	 We start by fixing a basis of $\mc{L}(D)$. We assume that the points $R_{\sigma^i}$ are pairwise distinct - it is easy to see that this assumption holds if the class $[P]\in E(L)/pE(L)$ is non-trivial.  To make the formulas simpler, we will assume $E$ is in short Weierstrass form, defined by $y^2=x^3+Ax+B$. Put $R_{g}=(x_{g},y_{g})$ for each non-trivial $g \in G$. Define $l_k=\frac{y+y_{\sigma^k}}{x-x_{\sigma^k}}$ for $1 \leq k \leq p-1$, and set $l_p=1$. 
		
		For $k<p$, it is clear that $l_k$ has a simple pole at $0_E$. We note that it has a simple pole at $R_{\sigma}$, and no other poles. Indeed, $x-x_{\sigma^k}$ is of degree two and vanishes at $R_{\sigma^k}$ and $-R_{\sigma^k}$, and $y+y_{R\sigma}$ vanishes at $-R_{\sigma^k}$. Now it follows easily that $l_k \in \mathcal{L}(D)$, and furthermore that $l_1,l_2,...,l_{p-1},l_p$ are linearly independent, and so they span the $p$-dimensional space $\mathcal{L}(D)$. 
		
		Note that it suffices to compute $M_{\sigma}$ and $M_{\tau}$, as $\sigma$ and $\tau$ generate $G$.
		
		\begin{proposition} \label{formula for M prop}
			The matrices $M_{\sigma}$ and $M_{\tau}$ in $\mathrm{GL}_p(L)$, relative to the basis $l_1,\ldots,l_p$, are given by 
			\[
			M_{\sigma}=\begin{pmatrix}
			0 & 0 & \hdots & 0 & -1 & 0 \\
			1 & 0 & \hdots & 0 & -1 & c_2 \\
			0 & 1  & \hdots & 0 & -1 & c_3 \\
			\vdots & \vdots & \ddots & \vdots & \vdots & \vdots \\
			0 & 0 & \hdots & 1 & -1 & c_{p-1} \\
			0 &  0 & \hdots & 0 & 0 & 1 
			\end{pmatrix}, 	\hspace{10pt}	M_{\tau}=\begin{pmatrix}
			0 & 0 & 0 & \hdots & 0 & 1 & 0 \\
			0 & 0 & 0 & \hdots & 1 & 0 & 0 \\
			\vdots & \vdots & \vdots & \ddots & \vdots & \vdots & \vdots \\
			0 & 1 & 0 & \hdots & 0 & 0 & 0 \\
			1 & 0  & 0 & \hdots & 0 & 0 & 0  \\
			0  & 0 & 0 & \hdots & 0 & 0 & 1 \\
			\end{pmatrix}
			\]
			where $c_k=\frac{y_{\sigma}+y_{\sigma^{k}}}{x_{\sigma}-x_{\sigma^{k}}}$ for $2 \leq k \leq p-1$. 
		
		\end{proposition}
		
	We need the following lemma.
		\begin{lemma} \label{lemma for computing Msigma}
			Let $E$ be an elliptic curve over a field $k$. For any $P=(x_P,y_P) \in E(k)$ different from $0_E$, let $l_P=\frac{y+y_P}{x-x_P} \in k(E)$. Let $P_1,P_2 \in E(k)$ be points such that $P_1 \ne -P_2$.
			
			\begin{enumerate}[label=(\roman*)]
			    \item 
		Define $f_{P_1,P_2}=\frac{\phi^{*}_{P_2}(l_{P_1+P_2})}{l_{P_1}} \in k(E)$. Then $f_{P_1,P_2}$ is regular at $P_1$ and we have $f_{P_1,P_2}(P_1)=1$.
		\item  For any $R \in E(k)$, we have $\frac{\phi^{*}_{R}(l_R)}{l_{-R}}=-1$.
		\item For distinct $R_1,R_2 \in E(k)$, we have $(l_{R_1}-l_{R_2})(0_E)=0$
		\end{enumerate}
		\end{lemma}
		\begin{proof}
		\begin{enumerate}[label=(\roman*)]
		    \item 
			It is clear that $f_{P_1,P}$ is regular at $P_1$.  We define a rational function $g \in k(E)$ by $P \mapsto f_{P_1,P} (P_1)$. As $\frac{y+y_{P_1+P_2}}{x-x_{P_1+P_2}}$ has a simple pole at $P_1+P_2$, the function $\phi^{*}_{-P_2}(\frac{y+y_{P_1+P_2}}{x-x_{P_1+P_2}})$ has a simple pole at $P_1$, and therefore $g$ is regular with no zeros on the open set $E\setminus \{-P_1\}$. But the only such rational functions are the constant ones, and since $g(0_E)=1$, we deduce that $g=1$, and hence $f_{P_1,P_2}(P_1)=1$.
			
		   \item Note that both $\phi^{*}_{R}(l_R)$ and $l_{-R}$ have simple poles at $-R$ and $0_E$. The Riemann-Roch space $\mathcal{L}((0_E)+(-R))$ is 2-dimensional, and therefore there exists a $c_R \in k$ such that $\phi^{*}_{R}(l_R)+c_R\cdot l_{-R}$ is a constant function, i.e. the function on $E(k) \times E(k)$
			\[
			(P,R) \mapsto \frac{y_{P+R}+y_{R}}{x_{P+R}-x_{R}}+c_R\frac{y_{P}-y_{R}}{x_{P}-x_{R}},
			\]
			depends only on $R$, and so can be viewed as a rational function on $E$. It is clearly regular on $E\setminus\{0_E\}$, and therefore it must be constant. Since it does not have a pole at $0_E$, we have $c_R=1$.
		    \item We compute
	\begin{equation*}
		\begin{split}
		& l_{R_1}-l_{R_2}=\frac{y+y_{R_1}}{x-x_{R_1}}-\frac{y+y_{R_2}}{x-x_{R_2}}=\frac{(y_{R_1}+y_{R_2})x-(x_{R_1}+x_{R_2})y-(y_{R_1}x_{R_2}+x_{R_1}y_{R_2})}{(x-x_{R_1})(x-x_{R_2})}.
		\end{split}
		\end{equation*}
			The numerator has a pole of order 3 at $0_E$, the denominator has pole of order 4, and hence $l_{R_1}-l_{R_2}$ vanishes at $0_E$,
		\end{enumerate}
		\end{proof}
		\begin{proof}[Proof of Proposition \ref{formula for M prop}]
		    The matrix $M_{\sigma}$ is determined by the equation $\phi^{*}_{R_{\sigma}} (l)=M_{\sigma} \cdot  \sigma(l)$, where $l$ is the column vector $(l_1,\ldots,l_p)^{T}$. By Lemma \ref{cocycle point lemma}, we have $R_{\sigma^{p-2}}=\sum_{i=0}^{p-2} \sigma^i(R_{\sigma})=-\sigma^{p-1}(R_{\sigma})$  and $\sigma(R_{\sigma^k})=R_{\sigma^{k+1}}-R_{\sigma}$.  As $l_p=1$, we have $\sigma(l_p)=\phi^{*}_{-R_{\sigma}}(l_p)=1$, giving the last row of $M_{\sigma}$. The proposition amounts to proving that for $1 \leq k \leq p-2$ we have
		\[
		\phi^{*}_{R_{\sigma}}(l_{k+1})=\sigma(l_k)-\sigma(l_{p-1})+c_k l_p,
		\]
		as well as 
		\[
		\phi^{*}_{R_{\sigma}}(l_{1})=\sigma(l_{p-1}).
		\]
		Suppose first that $k<p-1$. Recall that we have assumed that the points $R_g$ are distinct, and note that $\phi^{*}_{R_{\sigma}}(l_{k+1})$ has simple poles at $R_{\sigma^{k+1}}-R_{\sigma}=\sigma(R_{\sigma^k})$ and $-R_{\sigma}$, $\sigma(l_k)$ has simple poles at $\sigma(R_{\sigma^k})$ and $0_E$, and $\sigma(l_{p-1})$ has simple poles at $\sigma(R_{\sigma^{p-1}})=-R_{\sigma}$ and $0_E$.
		
		Hence the function $\frac{\phi^{*}_{R_{\sigma}}(l_{k+1})}{\sigma(l_k)}$ is regular at the point $\sigma(R_{\sigma^k})$. By Lemma \ref{lemma for computing Msigma}, taking $P_1=\sigma(R_{\sigma^k})=R_{\sigma^{k+1}}-R_{\sigma}$ and $P_2=-R_{\sigma}$, we see that $\frac{\phi^{*}_{R_{\sigma}}(l_{k+1})}{\sigma(l_k)}(\sigma(R_{\sigma^k}))=1$.. As a consequence we see that the function $\phi^{*}_{R_{\sigma}}(l_{k+1})-\sigma(l_k)$ is regular at $\sigma(R_{\sigma^k})$. By Lemma \ref{lemma for computing Msigma}(iii),  $\sigma(l_k)-\sigma(l_{p-1})$ is regular at $0_E$ with $\sigma(l_k-l_{p-1})(0_E)=0$. Hence the rational function
		$\phi^{*}_{R_{\sigma}}(l_{k+1})-\sigma(l_k)+\sigma(l_{p-1})$ has no poles except perhaps a simple one at $-R_{\sigma}$, and therefore must be the constant $-c_k$. By evaluating at $0_E$, we find that 
		\[
		c_k=-(\phi^{*}_{R_{\sigma}}(l_{k+1})-\sigma(l_k)+\sigma(l_{p-1}))=-(\phi^{*}_{R_{\sigma}}(l_{k+1}))(0_E)=-l_{k+1}(R_{\sigma})=-\frac{y_{\sigma^{k+1}}+y_{\sigma}}{x_{\sigma}-x_{\sigma^{k+1	}}}.
		\]
		as desired. To prove that $\sigma(l_{p-1})=-\phi^{*}_{R_{\sigma}}(l_{1})$, note that in the notation of Lemma \ref{lemma for computing Msigma}, $\sigma(l_{p-1})=l_{-R_{\sigma}}$, and $l_{1}=l_{R_{\sigma}}$, and so we are done by the final assertion of Lemma \ref{lemma for computing Msigma}.
		
		The computation of $M_{\tau}$ is simpler. Note that $R_{\tau}=0_E$, and so the last row of $M_{\tau}$ is the assertion that $\tau(l_p)=l_p=1$. For the other rows, we need to show that $\tau(l_k)=l_{p-k}$. This follows from the identity $\tau(R_{\sigma^k})=R_{\sigma^{p-k}}$, which is the content of Lemma \ref{cocycle point lemma}(iii).
		\end{proof}
		The cocycle condition determines the other $M_{g}$ as follows: $M_{\sigma^k}=M_{\sigma}\sigma(M_{\sigma})\cdots\sigma^{k-1}(M_{\sigma})$, and $M_{\sigma^k\tau}=M_{\sigma^k}M_{\tau}$.

		\section{Minimization and reduction} \label{minred koly section}
		Proposition \ref{Galois descent prop} reduces the problem of computing a $p$-diagram representing the Kolyvagin class to linear algebra over $\Q$, since we now only need to compute a basis for the $p$-dimensional $\Q$-vector space $\mc{L}(D)^G$.  However, if we do not do this linear algebra carefully, the resulting diagram $C \subset \mb{P}^{p-1}$ will be defined by equations with enormous coefficients. Moreover, even just doing this linear algebra can be computationally very expensive.  
		
		The reason for this is that the Heegner point we start with is typically of very large height.  From the theory of \textit{minimization} and \textit{reduction} of genus one models, as developed in \cite{minred234}, \cite{minred5} and \cite{LazarThesis}, we know that every element of the $n$-Selmer group of $E$ can be represented by a minimal model, which is an $n$-diagram given by equations with 'nice' equations. To make this more precise, Theorem 1.2 of \cite{ncovbds} shows that the coefficients of these equations are integers bounded by a power of the naive height of $E$, for $2 \leq n \leq 4$. Another result that is similar in spirit, that holds all odd $n$, is Theorem 1.0.1 of \cite{LazarThesis}. 
		
	 We are free to modify a $p$-diagram $[C \xrightarrow[]{} \PP^{p-1}]$ by making a linear change of coordinates on $\PP^{p-1}$. In this section we explain how to choose such a coordinate change so that the Kolyvagin class $c$ is represented by a diagram defined by equations with small integer coefficents. This breaks up into two steps known as minimization and reduction. The minimization step finds a $\mathrm{GL}_p(\Q)$-transformation so that the diagram $[C \xrightarrow[]{} \PP^{p-1}]$ can be represented by an integral model which has nice reduction properties modulo each prime. The reduction step then finds a $\mathrm{GL}_p(\Z)$-transformation to make the coefficients of such a model as small as possible.

\subsection{Minimization}		
    \subsubsection{Toy example}
	We give first an informal overview of what goes wrong with the naive approach and how it can be fixed. Consider the following simpler problem. Let $E/\Q$ be an elliptic curve, let $n \geq 5$ be an odd integer, let $[P] \in E(\Q)/nE(\Q)$, and suppose we want to compute an $n$-diagram representing the class $\delta([P]) \in H^1(\Q,E[n])$. Let $y^2=x^3+ax+b$ be a Weierstrass equation $W$ for $E$, with $a,b\in \Z$.
	
	As explained in Section \ref{koly background}, we need to choose a degree $n$ divisor $D$ with $\mathrm{sum}(D)=P$ and compute a basis for the Riemann-Roch space $\mc{L}(D)$. A natural choice would be to take $D=(n-1)\cdot (0_E) + (P)$, and for the basis $l_1,\ldots,l_{n-1}$ of $\mc{L}(D)$ we can take $1,x,y,x^2,xy,\ldots,x^{\frac{n-1}{2}}$ together with $\frac{y+y_P}{x-x_P}$, if $P=(x_P,y_P) \in E(\Q)$, and $\delta([P])$ is represented by $C \subset{\PP}^{n-1}$, where $C$ is the image of $E$ under the map $Q \mapsto (l_1(Q):\ldots:l_n(Q))$. An integral model for $C$ is then determined by $n(n-3)/2$ quadrics that form the basis for the $\mb{Z}$-module of quadrics in integral coefficients that vanish on $C$.  
	
	The problem that arises is that, if $q$ is a prime of good reduction of $E$, for which the point $P$ maps to zero under the reduction map $E \xrightarrow{} \Tilde{E}$, i.e. $q$ divides the denominators of $x_P$ and $y_P$, then the above basis does not reduce to a basis of $\mc{L}(\widetilde{D})$, and it is not difficult to show that this implies that the integral model for $C$ reduces to a singular curve modulo $q$. If the point $P$ is of large height, then the primes $q$ can be very large, and in practice this forces the coefficients of $C$ to be large, since the discriminant invariant of the integral model of $C$, as defined in \cite{minred234}, \cite{minred5} when $n \leq 5$, is a non-zero integer divisible by $q$, and so at least $q$.
	
	In this case, the issue can be resolved as follows. We first replace $(n-1) 0_E+P$ by the linearly equivalent divisor $(n+1) \cdot (0_E)-(-P)$. Then $\mc{L}(D) \subset \mc{L}((n+1) \cdot 0_E)$. It follows from the Riemann-Roch theorem that $1,x,y,x^2,xy,x^{\frac{n-1}{2}},x^{\frac{n-3}{2}}y$ is a basis of $\mc{L}((n+1) \cdot 0_E)$. This is a nice basis, in the sense that if $q$ is a prime that does not divide the discriminant of $W$, then $1,\Tilde{x},\Tilde{y},\Tilde{x}^2,\Tilde{x}\Tilde{y},x^{\frac{n-1}{2}},\Tilde{x}^{\frac{n-3}{2}}\Tilde{y}$ is a basis of $\mc{L}(n \cdot 0_{\Tilde{E}})$. For the basis of $\mc{L}(D)$ we take $l_1,\ldots,l_{n}$ to be a $\Z$-basis of the module of spanned by those $\Z$-linear combinations of $1,x,y,x^2,xy,x^{\frac{n-1}{2}},x^{\frac{n-3}{2}}y$ that vanish at $-P$. Then $l_1,\ldots,l_{n}$ reduces to a basis of $\mc{L}(\Tilde{D})$ on the reduction $\Tilde{E}$. The curve $C \subset \PP^{n-1}$ defined by this embedding  has an integral model that reduces to a non-singular curve modulo any prime $q$ with $q \nmid \mathrm{Disc}(E)$, see \cite[Lemma~7.4.5]{LazarThesis}. In fact, if the Weierstrass equation $W$ is minimal, one can also show that this model is a minimal genus one model, in the sense of Theorem 4.1.1 of \cite{LazarThesis}.

	\subsubsection{Minimizing the Kolyvagin class}	Let us recall the setup of Section \ref{ndiag koly section}. We assume that we have the data specified in Lemma \ref{cocycle point lemma}: an elliptic curve $E/\mathbb{Q}$, an odd prime $p$, an imaginary quadratic field $K/\mathbb{Q}$, a cyclic extension $L/K$ of degree $p$, and a point $P \in E(L)$ such that $[P] \in (E(L)/pE(L))^G$, where $G=\mathrm{Gal}(L/\mathbb{Q})$. We also assume, for simplicity, that $L$ has class number one. This assumption holds in most of the examples we were able to compute in practice. From this data one can compute the points $R_{g} \in E(L)$ giving rise to the cocycle $g \mapsto R_g$ defined in Lemma \ref{cocycle point lemma}. 
		
		Fix a short Weierstrass equation $W$ of $E$, and let $R_g=(x_g,y_g)$ for $g \in G$. Recall that we have defined two divisors on $E$ 
		\[
		D=0_E+(R_{\sigma})+...+(R_{\sigma^{p-1}})  ,
		\]
		\[
		D'=(2p-1) \cdot 0_E-(-R_{\sigma})-...-(-R_{\sigma^{p-1}}).
		\]
    	
		Let $D'$ be the divisor $(2p-1) \cdot 0_E-(-R_{\sigma})-...-(-R_{\sigma^{p-1}})$. As $\mathrm{sum}(D')=\mathrm{sum}(D)$ and both $D'$ and $D$ have degree $p$, divisors $[D]$ and $[D']$ are linearly equivalent, and the Kolyvagin class $c_L \in H^1(L,E[p])$ is represented by the $p$-diagram $[E \xrightarrow[]{} \mb{P}^{p-1}]$ where the map is induced by the linear system $|D'|$. A choice of a rational function $f$ with $\mathrm{div}(f)=D-D'$ determines an isomorphism between the vector spaces $\mathcal{L}(D)$ and $\mathcal{L}(D')$, and we transport the action of $G$ to $\mathcal{O}_E(D')$ via such an isomorphism. The subset $\mc{L}(D')^{G}$ of $G$-invariant elements is a $p$-dimensional $\Q$-vector space. We can naturally view $\mathcal{L}(D')$ as a subspace of  $\mathcal{O}_E((2p-1) \cdot 0_E)$ that consists of functions that vanish at $-R_{\sigma},-R_{\sigma^2},...,-R_{\sigma^{p-1}}$.

		Note that $1,x,y,x^2,xy,...,x^{p-1},x^{p-2}y$ is a basis of $\mathcal{L}((2p-1) \cdot 0_E)$. Let $S \subset \mathcal{L}((2p-1) \cdot 0_E)$ be the free $\mathcal{O}_L$-module spanned by $1,x,y,x^2,xy,...,x^{p-1},x^{p-2}y$ and let $T=S \cap \mathcal{L}(D')$. Then the subset $T^{G}$ of invariant elements of $T$ is a $\mb{Z}$-module, and since we have $T^{G} \otimes \Q=\mc{L}(D')^{G}$, it is a free $Z$-module of rank $p$.
		
		Finally, let $l_1,\ldots,l_p$ be a basis of $T^{G}$ as a $\Z$-module. Then $l_1,\ldots,l_p$ is also a basis of $\mc{L}(D')^{G}$ as a $\Q$-vector space, and hence, by Proposition \ref{Galois descent prop}, determines an embedding $C \xrightarrow{} \PP^{p-1}$ that represents the Kolyvagin class $c$. This representation of $c$ is sufficiently nice for our purposes.   In fact, in \cite[Section~7.4]{LazarThesis}, we show that the reduction of this curve mod $q$ is non-singular for any prime $q$ that does not divide the discriminant of $E$ or the discriminant of $L$, and we conjecture that the resulting genus one model is minimal.

		\subsubsection{Practical computation of a minimal model} \label{practical minimization}
	    We now explain how to compute equations for the $p$-diagram defined above. This amounts to doing the linear algebra of Proposition \ref{Galois descent prop} over $\Z$, in a suitable sense.  We need to take some care with solving the resulting linear equations, since they have very large coefficients.
		
		\textbf{Matrix representation of a basis of $\mathcal{L}(D').$}	We have an inclusion $\mathcal{L}(D') \subset \mathcal{L} ((2p-1) \cdot 0_E)$. Let $e_1=1,e_2=x,e_3=y,...,e_{2p-2}=x^{p-1},e_{2p-1}=x^{p-3}y$ be the standard basis of $\mathcal{L}((2p-1)\cdot 0_E)$. For a basis $f_1,...,f_p$ of $\mathcal{L}(D')$, we can then write $f_i=\sum_{j=1}^{2p-1} A_{ij} e_{j}$ for some $A_{ij} \in L$, and from now on we identify the vector space $\mathcal{L}(D')$ with the span of the rows of $A$.

		Recall that we have defined a free $\mathcal{O}_{L}$-module $T$ as the subset of elements of $\mathcal{L}(D')$ that can be expressed as $\mathcal{O}_{L}$-integral linear combinations of $e_1,...,e_{2p-1}$. Under the above identification, elements of $T$ correspond to linear combination of rows with integral entries. 
		
		\textbf{Making the action of $\mathrm{Gal}(L/\mathbb{Q})$ explicit.}
		In Section 2 we have defined a basis of $\mathcal{L}(D)$ by setting $l_1=\frac{y+y_{\sigma}}{x-x_{\sigma}},...,l_{p-1}=\frac{y+y_{\sigma^{p-1}}}{x-x_{\sigma^{p-1}}},l_p=1$. For the function $f$ with $\mathrm{div}(f)=D-D'$, we take $f=\prod^{p-1}_{k=1}(x-x_{\sigma^{k}})$. Then $l'_i=f \cdot l_i$, $1 \leq i \leq p$, is a basis of $\mathcal{L}(D')$. Let $\alpha_1,...,\alpha_{2p}$ be a $\mathbb{Z}$-basis of $\mathcal{O}_{L}$, and consider $\mathcal{L}(D')$ as a $\mathbb{Q}$-vector space with the basis $\alpha_i f_j$, $1 \leq i \leq 2p$, $1 \leq j \leq p$. Recall the matrices $M_g \in \mathrm{GL}_p(L)$ computed in Proposition \ref{formula for M prop}. Let $N_g=M_g^{-1}$. By Proposition \ref{galois action divisor prop}, we have
		
		\[  \begin{pmatrix}
		g(l'_{1}) \\
		g(l'_{2}) \\
		\vdots \\
		g(l'_{p})
		\end{pmatrix}
		= N_{g} \cdot
		\begin{pmatrix}
		l'_{1} \\
		l'_{2} \\
		\vdots \\
		l'_{p}
		\end{pmatrix}.
		\]
		As the action of $G$ is semilinear, by multiplying the left and right sides by $g(\alpha_j)$ we obtain 
		\[  \begin{pmatrix}
		g(\alpha_j l'_{1}) \\
		g(\alpha_j l'_{2}) \\
		\vdots \\
		g(\alpha_j l'_{p})
		\end{pmatrix}
		=\frac{g(\alpha_j)}{\alpha_j} N_{g} \cdot
		\begin{pmatrix}
		\alpha_j l'_{1} \\
		\alpha_j l'_{2} \\
		\vdots \\
		\alpha_j l'_{p}
		\end{pmatrix}.
		\]
		For a general basis $f_1,...,f_p$ of $\mathcal{L}(D)$ we have a similar formula, with $N_g$ represented by $g(T)N_gT^{-1}$, where $T\in\mathrm{GL}_{p}(L)$ is the matrix relating the bases $l'_i$ and $f_i$.  	
		
		\textbf{Computing an integral basis of $T^{G}$.}
		Using the above formulas we can compute the row vector representation of $g(\alpha_i f_j)$ for each $i$ and $j$. Let $A_1 \in \mathrm{Mat}_{2p^2,2p-1}(L)$ be the matrix formed by the rows corresponding to the elements $\sum_{g \in G} g(\alpha_i f_j)$ for $1 \leq i \leq 2p$, $1 \leq j \leq p$. By Proposition \ref{galois action divisor prop}, the space $\mathcal{L}(D')^{G}$ is spanned by the rows of $A_1$. Note that $T^{G}=\mathcal{L}(D')^{G} \cap T$. 
		
		Next, the choice of basis $\alpha_1,...\alpha_{2p}$ determines an isomorphism $L \cong \mathbb{Q}^{2p}$, and hence an isomorphism $\mathrm{Mat}_{k,l}(L) \cong\mathrm{Mat}_{k,2pl}(\mathbb{Q})$ for any pair $k,l$. To be explicit, for each entry $z$ of $A_1$,  write $z=c_1\alpha_1+...+c_{2p}\alpha_{2p}$ for $c_i \in \mathbb{Q}$, and let $A_2 \in \mathrm{Mat}_{2p^2,2p(2p-1)}(\mathbb{Q})$ be the matrix obtained from $A_1$ by replacing each entry of $A$ by the associated $2p$-tuple $c_1,...,c_{2p}$. We say $A_2$ is obtained from $A_1$ by the restriction of scalars from $L$ to $\mathbb{Q}$.
		
		The space $T^G$ then corresponds to the $\mathbb{Z}$-sublatice of row vectors with integral entries in the $\mathbb{Q}$-span of rows of $A_2$, and finding a basis for such a lattice is a standard problem, which can be solved efficiently using Hermite normal form.

		We summarise the above discussion in the following algorithm:
		
		\begin{algorithm} \label{mini algortihm}
				\begin{itemize}
				\item[]
				\item INPUT: $E,D$, $p \geq 5$ and a point $P \in E(L)$ that satisfies the conditions of Lemma \ref{cocycle point lemma}.
				\item OUTPUT:  An integral model $\mathcal{C}$ of the class $c_{\mathbb{Q}}$.
			\end{itemize}  			
		\begin{enumerate}
			\item Compute the points $R_{g}$, and then the matrices $M_g \in \mathrm{GL}_p(L)$ using Proposition \ref{formula for M prop}.
			\item Choose a basis $f_1,...,f_p$ of $\mathcal{L}(D')$, and represent it by a matrix $A \in \mathrm{Mat}_{p,2p-1}(L)$. Use the formulas defining the Galois action to compute the matrix $A_1 \in \mathrm{Mat}_{2p^2,2p-1}(L)$ representing a set of generators of $\mathcal{L}(D')^G$, and its restriction-of-scalars representation $A_2 \in \mathrm{Mat}_{2p^2,2p(2p-1)}(\mathbb{Q})$, as described above. Let $V$ be the $\mathbb{Q}$-span of rows of $A_2$, and set $T'=\mathbb{Z}^{2p(2p-1)} \cap V$.
			\item Compute a basis of $T'$ as a matrix  $B' \in \mathrm{Mat}_{p,2p(2p-1)}(\mathbb{Z})$, and then compute the matrix $B \in \mathrm{Mat}_{p,2p-1}({\mathcal{O}_{L}})$ such that $B'$ is the restriction of scalars of $B$. We recover a basis of $T$ by setting $f^{G}_i=\sum_{j=1}^{2p-1}B_{ij} e_j$ for $1 \le i \le p$.
			\item Compute a basis $q_1,...,q_{p(p-3)/2}$ for the $\mathbb{Z}$-module of quadrics with $\mathbb{Z}$-coefficients vanishing on the image of $E$ in $\mathbb{P}^{p-1}$ under the map $e$ induced by $f^{G}_1,...,f^{G}_p$, and return as the model $\mathcal{C}$ the subscheme of $\mathbb{P}_{\mathbb{Z}}^{n-1}$ defined by the $q_i$.
		\end{enumerate}.
		\end{algorithm}	
	
		Two steps of the algorithm need further explanation. We need to explain how to compute the quadrics in Step (iv), which is straightforward and we  do first, and we need explain how to choose the basis in Step (ii), which is a subtler problem.

		\textbf{Step 4 of the algorithm.} Let $C_{\mathbb{Q}}$ be the image of $E$ in $\PP^{p-1}$ under the embedding $e$ and let $C_{L}$ be the base change of $C_{Q}$ to $L$. As $C_L$ is a genus one normal curve, so in particular projectively normal, the monomials $f^G_if^G_j$, $1 \leq i,j \leq p$, span the $2p$-dimensional $L$-vector space $\mathcal{L}(2D')$.
		
		Let $x_1,...,x_p$ be the coordinates on $\mathbb{P}^{p-1}$, and let $V_{\mathbb{Q}}$  be the $\mathbb{Q}$-space of all rational quadratic forms, spanned by the monomials $x_ix_j$, $1 \leq i,j \leq p$. We then define a $\mathbb{Q}$-linear map $j: V \xrightarrow{} \mathcal{L}(2D')$ by the rule $x_i x_j  \mapsto f^G_i f^G_j$. 
		
		The kernel of this map consists of all of the quadrics that vanish on $C_{\mathbb{Q}}$. We compute a matrix representing the map $j$, and then use linear algebra over $\mathbb{Z}$ to compute a set of generating quadrics $q_1,...,q_{p(p-3)/2}$ of $I(C_{\mathbb{Q}})$, with the property that they generate the $\mathbb{Z}$-submodule of integral quadrics that vanish at $C_{\mathbb{Q}}$.
		
		 When $p=3$, $C_{\mathbb{Q}}$ is defined by a single ternary cubic, and it is simple to adapt the above method to work in this case as well.

		\textbf{Picking a basis in Step 2.} A natural choice of basis of $\mathcal{L}(D')$, given the computation of Proposition \ref{formula for M prop}, would be $l'_1,...,l'_p$. This, however, does not lead to a practical algorithm. With this choice, computing the basis of $T'$ in Step 3 can be very time consuming, as the dimension of matrix $A_2$ grows quickly with $p$ and the entries of $A_2$ tend to be rational numbers of large height, as they were obtained from the coordinates of the Heegner point. We now describe a more careful way to choose a basis.
		
		We start by rescaling the basis $l'_i$. For legibility write $x_i=x_{\sigma^i}$ and $y_{\sigma^i}=y_i$. As $\mathcal{O}_{L}$ is a PID, it is a standard fact that we can write $x_i=\frac{r_i}{t_i^2}$ and $y_i=\frac{s_i}{t_i^3}$ for some  $r_i,s_i,t_i$, with $r_i,t_i$ and $s_i,t_i$ being pairs of coprime algebraic integers. For $1 \leq i \leq p-1$, we put
		
		\[
		f'_i=t_i \cdot t_{1}^2\cdots t^2_{p-1} \cdot l'_i=(t_1^2x-r_1)\cdots(t_{i-1}^2+r_{i-1})(t_i^3y+s_i)(t_{i+1}^2x+r_{i+1})\cdots(t_{p-1}^2x-r_{p-1}) 
		\]
		Having chosen this scaling, we see that $f'_i \in T$, i.e. the matrix $A'$ whose rows represent $f'_i$ have integral entries, and so do the corresponding matrices $A'_1$ and $A'_2$.
		
		 Our next step is motivated by the following heuristic. For $l<k<p$, we have $R_{\sigma^k}-R_{\sigma^l}=\sum^{k}_{i=1} \sigma^{i-1}(R_{\sigma})-\sum^{l}_{i=1} \sigma^{i-1}(R_{\sigma})=\sigma^{l}(R_{\sigma^{k-l}})$. If, for a prime $\gp$ of $\mathcal{O}_{L}$, the point $\sigma^{l}(R_{\sigma^{k-l}})$ reduces to $0_{\widetilde{E}}$, then $\widetilde{R_{\sigma^k}}=\widetilde{R_{\sigma^l}}$, and hence $\widetilde{f_k}=\widetilde{f_l}$. Hence $\gp$ will divide all entries of the difference $r_{k,l}$ of rows of $A'$ corresponding to $f_k$ and $f_l$. As the primes for which $\sigma^l(R_{\sigma^{k-l}})$ reduces to zero are exactly those that divide $\sigma^l(t_{k-l})$, we expect that the entries of $r_{k,l}$ and $\sigma^{k-l}(t_l)$ will have a large common divisor.
		
		To cancel out these divisors for all pairs of rows we use the following procedure, reminiscent of Gaussian elimination.  Let $r_1,...,r_p$ be the rows of $A'$. For $1 \leq  k \leq p-1$, consider the $2 \times (2p-1)$ submatrix $A^{k}$ of $A'$ formed by $r_k$ and $r_{p-1}$, and let $d_k$ be the generator of the ideal of $\mathcal{O}_{L}$ generated by the $2\times2$ minors of $A^{k}$. We then compute $c_k \in \mathcal{O}_{L}$ such that the entries of $r_{p-1}-c_{k}r_{k}$ are divisible by $d_k$ - this amounts to putting $A^k$ in Hermite normal form (over $\mathcal{O}_L$), which is possible since we assumed $\mathcal{O}_{L}$ is a PID, and can be done efficiently.
		
		 Next, compute a generator $D_k$ of the ideal $(d_{k},d_1\cdots d_{k-1}d_{k+1}\cdots d_{p-2})$, and find $i_k \in \mathcal{O}_{L} \frac{d_{k}}{D_{k}}$ and $j_k \in \mathcal{O}_{L} \cdot \frac{d_{k},d_1\cdots d_{k-1}d_{k+1}\cdots d_{p-2}}{D_{k}}$, with $i_k+j_k=1$ - this is also a standard problem, see  \cite{cohen1996hermite}. We then replace $r_{p-1}$ with $r'_{p-1}=r_{p-1}-j_1c_1 \cdot r_1-...-j_{p-2}c_{p-2} \cdot r_{p-2}$, and then divide $r'_{p-1}$ by the GCD of its entries. In practice, $D_k$ will often be a unit or at worst divisible by a few small primes, and so this GCD will be the product $d_{1,p-1}\cdots d_{p-2,p-1}$, up to a small factor. We then repeat this process for rows $r_1,...,r_{p-2}$, with $r_{p-2}$ taking the role of $r_{p-1}$, and so on.  
		
		At the end of this process we obtain a new matrix $A'' \in \mathrm{Mat}_{p,2p-1}(\mathcal{O}_{L})$ and $U \in \mathrm{GL}_{p}(L)$ with $A''=UA'$. We then take $f_1,...,f_p$ to be the basis of $\mathcal{L}(D')$ that corresponds to the rows of $A''$. To account for the change of basis, we replace $M_{g}$ by $UM_{g}g(U^{-1})$ for each $g \in G$, and then compute a basis $f^{G}_1,...f^{G}_p$ of $T^{G}$ using the approach described for $A$.
		
			\subsection{Reduction} \label{reduction section}
	The final step is to find a $\mathrm{GL}_{p}(\Z)$-change of coordinates making the coefficients of the equations defining $C \subset \mb{P}^{p-1}$ as small as possible. For this, we use the method of reduction, developed in Section 6 of \cite{minred234}. This method extends with minimal changes to our setting, so we give a very brief summary. 
	
	To compute a reduced $p$-diagram equivalent to the diagram $[C \subset{} \PP^{p-1}]$ representing the Kolyvagin class, we first compute the reduction covariant $\phi(C)$, according to the recipe given in Section 6 of \cite{minred234}. The reduction covariant is a certain symmetric positive definite matrix, well-defined up to a scalar in $\mb{R}^{\times}$, one associates to a $p$-diagram $[C \subset \PP^{p-1}]$ defined over $\mb{R}$, which transforms in a natural way under linear changes of coordinates on $\PP^{p-1}$.  Computing it amounts to computing the set of flex points of $C$, i.e. points $P \in C(\mb{C})$ with the property that the tangent hyperplane at $P$ meets $C$ only at $P$. In our case this is is easy, since we have a description of $C$ as an embedding of $E$ via the complete linear system $|D'|$. We then use the LLL algorithm to compute a $g \in SL_n{\mb{Z}}$ such that $g^{-t} \phi(C) g^{-1}$ is LLL-reduced, and replace $C$ with $g(C)$. For more details on our implementation, see Section 7.4.3 of \cite{LazarThesis}.

		\section{Examples} \label{koly example section}
		
		In this section we apply the theory we developed to concrete examples, and construct elements of $p$-torsion subgroups of Tate-Shafarevich groups, for $p \leq 11$ an odd prime. 
		
		As mentioned in the introduction, these computations are of the most interest when the prime $p$ is at least 7, since for $p \leq 5$ the usual method of $p$-descent works quite well for computing these examples. As a warmup, we compute an element of $\Sha(E/\Q)[3]$ for the curve $E$ labelled $681b3$ in Cremona's tables. We then follow with our main result, explicit equations representing an element of $\Sha(E/\mathbb{Q})[p]$ for $p=7$, where $E$ is the curve $3364c1$. 
		
		 Note that we can't apply the theory we developed in Section \ref{koly background} to   the Kolyvagin class $c_{\Q}$ directly. We defined this class as the image of the class $[D_{\sigma} z_K] \in E(L)/pE(L)$, and the point $D_{\sigma} z_K$ need not satisfy the conditions of Proposition \ref{cocycle point lemma}, since it is not necessarily fixed by complex conjugation $\tau$. However, since $c_{\Q}$ is in the $\pm$-eigenspace of $H^1(\Q,E[p])$, where $\pm$ is the sign of the functional equation of $E$,  we have $2\cdot [D_{\sigma} z_K]=[D_{\sigma} z_K \pm \tau D_{\sigma} z_K] \in (E(L)/pE(L))^{G}$, and so the point $P=D_{\sigma} z_K \pm \tau D_{\sigma} z_K$ satisfies the conditions of Prop. \ref{cocycle point lemma}, and we compute the Kolyvagin class associated to this point.  We can then recover the $\mb{F}_{p}$-line spanned by $c$ in $H^1(\Q,E[p])$, since we can use Algorithm \ref{mini algortihm} to compute a $p$-diagram representing the class $[mP]$ for any $m \in \Z$.
		 
		 Recall that by Proposition \ref{cocycle point lemma}(ii), we have $[P]=[D_{\sigma}R_{\sigma}]$, with $R_{\sigma}=\frac{\sigma(P)-P}{p}$. In practice, the point $R_{\sigma}$ is of much smaller height than $z_K$, and it is simple to adapt Algorithm \ref{heegner point algorithm} to compute this point directly. However, a drawback is that there are $p^2$ possible $p$-division points of the point $\sigma(P)-P$ in $E(\mb{C})$, and only one of them is the point $R_{\sigma}$. Thus we use Algorithm \ref{heegner point algorithm} on each division point successively, until the algorithm returns a point. If the height of the Heegner point is very large, like in our $p=7$ example, then this is worth doing, however if $p$ is large and the height of the point is small, like in our $p=11$ example, then we compute $P$ first to avoid slowing down the code.

		 For all of our computations we have used the computer algebra system MAGMA (\cite{bosma1997magma}). The source code, along with further examples we have computed, is available as a GitHub repository at \url{https://github.com/lazaradicevic/kolyvagin.classes}.

		\begin{example}
			Consider the elliptic curve $E$ 
		labelled 681b3 in Cremona's tables. $E$ has no rational 3-isogeny 
		and $\Sha(E)[3]=(\mathbb{Z}/3\mathbb{Z})^2$. There are 
		no elliptic curves of smaller conductor with this property, so $E$ is a 
		natural first candidate for us. $E$ is defined by the minimal Weierstrass equation
		
		\[
		y^2 + xy = x^3 + x^2 - 1154x - 15345
		\]
		
		 For our Heegner discriminant, we choose $D=-107$. The conductor of $E$ is $N=3\cdot 227$, and one verifies that $3$ and $227$ split completely in $K=\mathbb{Q}(\sqrt{-107})$, so $D$ satisfies the Heegner hypothesis. 
		
		For our field $L$, we take the Hilbert class field of $K$. As $K$ has 
		class number 3, by class field theory $L/\mathbb{Q}$ is a dihedral 
		extension of degree 6. We use the machinery implented in MAGMA to find 
		that $L=\mathbb{Q}[\alpha]$, where the minimal polynomial of $\alpha$ 
		is $x^6 - 2x^5 - 2x^3 + 30x^2 - 52x + 29$, and that $L$ has class number 1. We fix an ideal $\mc{N}$ with $\mc{N}\bar{\mc{N}}=N\mc{O}_K=681\mc{O}_K$. Let $z_K \in E(H)$ be the Heegner point that is the image of the point $(\mc{O}_K,[\mc{O}_K],\mc{N})$. There are 4 possible choices for $\mc{N}$, corresponding to the factorization $N=3\cdot 227$. Which one we choose is not important for the purpose of constructing non-trivial Kolyvagin classes, since changing the choice of $\mc{N}$ replaces $z_K$ by $\pm z_K+T$, where $T \in E_{tors}(\mb{Q})$. See Proposition 5.3 of \cite{gross}.
		
Using Algorithm \ref{heegner point algorithm}, slightly modifed as explained above, we compute the point $R_{\sigma} \in E(L)$. Its $x$-coordinate  is 
		\begin{align*}
		&	1/1741682413263770958143450(483403026915311979182787081\alpha^5 + \\&
			35453825605498566073743810\alpha^4 - 137498458568104949011766487\alpha^3 -\\& 
			2452468960182058461987679215\alpha^2 + 9038525365115044024894770546\alpha - \\&
			3473956084362757366189406163).
		\end{align*}
Next, we run the Algorithm \ref{mini algortihm} up to Step 3, computing rational functions $l_1,l_2$ and $l_3$ that form an invariant basis of $\mc{L}(5\cdot 0_E-R_{\sigma}-R_{\sigma^2})$. Let $C$ be the image of $E$ in $\mathbb{P}^2$ under the map $(x,y) \mapsto (l_1(x,y) : l_2(x,y) :l_3(x,y))$. Step 4 of Algorithm \ref{mini algortihm} does not apply in this case, since $C$ is defined by a ternary cubic rather than by quadrics. However it is easy to compute a cubic $G$ defining $C$, using the standard algorithms implemented in MAGMA: 
\begin{align*}
G=& \ 2372x^3 + 4174x^2y - 3043x^2z + 2340xy^2 - 3457xyz + 1271xz^2 + \\& \ 
419y^3 - 940y^2z + 700yz^2 - 173z^3
\end{align*}
The next step is to reduce $G$. Making the change of coordinates corresponding to the matrix
\[
\begin{pmatrix}
   - 1& -5& -9\\
0&  3&  4\\
0&  1& 1
\end{pmatrix},
\]  
replaces the cubic $G$ by 
\[
F=x^3 + 2x^2y - 3x^2z - xy^2 + 9xyz - 8xz^2 + y^3 - 11y^2z - 5yz^2 +
6z^3
\]
The cubic $F$ is minimal, in the sense of Definition 3.1 of \cite{minred234}, and is essentially as nice as of an equation as we can hope. Let $C \subset \mathbb{P}^2$ be the curve defined by $F$. The diagram $[C \xrightarrow{} \mathbb{P}^2]$ is the representation of the class $c_{\mathbb{Q}} \in H^1(\mathbb{Q},E[3])$ that we set out to obtain.

Note that as Algorithm \ref{heegner point algorithm} does not prove that the point we computed is the Heegner point, see Remark \ref{remark verified heegner}, we still need to prove that this class is in the 3-Selmer group $\mathrm{Sel}^{(3)}(E/\mathbb{Q})$, i.e. that the curve $C$ is everywhere locally soluble. Since we have represented $c$ by the ternary cubic $F$, we can use the standard algorithms  for genus one models implemented in MAGMA to do so.

Finally, to check that the image of $c_{\mathbb{Q}}$ is a non-trivial element of $\Sha(E/\mathbb{Q})$, note that $E$ has rank 0 and $E(\mathbb{Q})[3]$ is trivial. Hence $E(\mathbb{Q})/3E(\mathbb{Q})$ is trivial, and we only need to show that $c_{\mathbb{Q}}$ is non-zero. Thus it suffices to check that the class $[D_{\sigma}R]$ is non-zero in $E(L)/3E(L)$, i.e. that $D_{\sigma} \cdot R_{\sigma}$ is not divisible by $3$, and it is easy to check that this is indeed the case. Hence $C(\mathbb{Q})$ is empty, and $C$ is a counterexample to the Hasse principle. 

\begin{remark} \label{capitulation 3 remark}
One can easily find an equation for $c_{\mb{Q}}$ using the method of $3$-descent. However, our method gives us an additional piece of information - we know that the class $c_{\mb{Q}}$ capitulates over the field $L$, i.e. the curve $C$ admits an $L$-rational point. By construction of our model for $[C \xrightarrow{} \mb{P}^{n-1}]$, we know that the images of the points $R_{\sigma^{i}}$, where $0 \leq i \leq p-1$, under the embedding $E \xrightarrow{l} \mb{P}^{p-1}$, lie on the intersection of the curve $C$ and a hyperplane $H$ defined over $\mb{Q}$. We can compute an equation for $H$ using linear algebra, since $p$ points uniquely determine a hyperplane in $\mb{P}^{p-1}$. In our case, we find 
\[
H= 771x -2818y+ 4751z.
\]
Hence, the binary cubic obtained by substituting $z=-771/4751x + 2818/4751y$ in $F$ splits as a product of three distinct linear forms over $L$. 
\end{remark}
 \end{example}

\begin{example} \label{7 example}
	Let $E$ be the curve labeled 3364c1 in Cremona's tables, defined by a minimal Weierstrass equation $y^2 = x^3 - 4062871x - 3152083138$. Similarly to the previous examples, we chose $E$ because it is the smallest rank 0 curve with no rational 7-isogeny and $\Sha(E/\mathbb{Q})[7] \cong (\mathbb{Z}/7\mathbb{Z})^2$.
	
	For the Heegner discriminant, we take $D=-71$, which has class number 7. The Hilbert class field $L$ of $K=\mb{Q}(\sqrt{-71})$ is a degree 14 dihedral extension of $\mb{Q}$, defined by
	\begin{align*}
	f= \ &x^{14} + 7x^{13} + 25x^{12} + 59x^{11} + 103x^{10} + 141x^{9} + 159x^{8} + 153x^7 + \\&
	129x^6 + 95x^5 + 58x^4 + 27x^3 + 10x^2 + 3x + 1;
	\end{align*}
	The class group of the Hilbert class field $L$ of $K$ is trivial, so $\mathcal{O}_{L}$ is a PID. We compute the point $R_{\sigma}$ directly, using 250 digits of precision in the computation of modular parametrization, and finding a point of height $194.99$. This calculation took less than a minute with our MAGMA implementation. Unfortunately the results of the computation are too large to give here, as the $x$-coordinate of $R_{\sigma}$ would take several pages to print,  so we refer the reader to our GitHub repository.

	As before, we use the algorithm of Section \ref{practical minimization} to compute the rational functions $l_1,l_2,\ldots,l_7$ that define an embedding $E \xrightarrow{} \mathbb{P}^6$ over $L$. The image of the embedding is a curve $C$, that admits over $\mathbb{Q}$. We compute a basis for the 14-dimensional space of quadrics that cut out  $C$ in $\mb{P}^{6}$. The coefficients of all of the equations are remarkably small. We give the equations below:

\small
\begin{align*}f_{1}&= 2x_1x_2 - 2x_1x_3 + 2x_1x_5 - x_1x_6 + x_2^2 - x_2x_3 - 
	x_2x_4 + 2x_2x_5 - 4x_2x_6 + 2x_2x_7 - 3x_3^2 - 3x_3x_5\\&
	- x_3x_6 + 4x_4^2 - 2x_4x_5 + 3x_4x_6 - x_4x_7 - 3x_5^2 - 
	4x_5x_6 - x_5x_7 + 2x_6^2 - x_6x_7 - 2x_7^2,\\
	f_{2}&=x_1x_3 + x_1x_4 + x_1x_5 + 2x_1x_6 - x_1x_7 + 4x_2x_4 - 
	x_2x_5 + 2x_2x_6 - 2x_2x_7 - x_3^2 + 4x_3x_4\\& - 2x_3x_5 +
	5x_3x_6 - 3x_3x_7 + x_4^2 + 2x_4x_5 - x_4x_6 - 4x_4x_7 +
	3x_5x_6 - 2x_5x_7 + 2x_6^2,\\f_{3}&=x_1^2 + 4x_1x_3 + x_1x_5 + x_1x_6 - 4x_1x_7 - x_2x_3 + 
	3x_2x_4 - 2x_2x_5 + x_2x_6 - x_2x_7 + x_3^2 - 2x_3x_4 + \\&
	x_3x_5 + x_3x_6 - x_3x_7 - x_4^2 + x_4x_5 + x_4x_6 + 
	3x_4x_7 + x_5^2 + 2x_5x_6 - 2x_5x_7 + 5x_6^2 - 5x_7^2,\\
	f_{4}&=x_1^2 + 4x_1x_3 + 2x_1x_4 + x_1x_5 + x_1x_6 - 6x_1x_7 + 
	x_2^2 + x_2x_3 + 2x_2x_4 + x_2x_5 + x_2x_6 + 2x_3^2 + \\&
	x_3x_4 + x_3x_5 + 2x_3x_6 - 3x_3x_7 + x_4^2 - 2x_4x_5 + 
	2x_4x_6 - 3x_4x_7 + 2x_5^2 + 2x_5x_6 + x_5x_7 + x_6^2 -\\& 
	4x_6x_7 - x_7^2,\\f_{5}&=x_1x_2 - 3x_1x_3 - x_1x_4 - x_1x_5 - 2x_1x_6 + 5x_1x_7 - 
	x_2^2 - x_2x_3 - 3x_2x_4 - x_2x_5 - 2x_2x_6 +\\& 3x_2x_7 + 
	x_3x_5 - 3x_3x_7 + 3x_4^2 - 3x_4x_6 - 2x_4x_7 - 2x_5^2 - 
	x_5x_6 - 5x_6^2 - 2x_6x_7 - 3x_7^2,\\f_{6}&=2x_1^2 + x_1x_2 + x_1x_3 + 2x_1x_4 - x_1x_6 - 5x_1x_7 + 
	2x_2x_4 + x_2x_5 + 4x_2x_6 - x_2x_7 - x_3^2 + 4x_3x_5 + \\&
	2x_3x_6 + 2x_3x_7 + 2x_4^2 - x_4x_5 + 2x_4x_6 + 3x_5^2 + 
	x_5x_6 - 4x_5x_7 - 3x_6^2 - x_6x_7 - 4x_7^2,\\
		f_{7}&=x_1^2 + x_1x_2 + x_1x_3 + 2x_1x_4 - 3x_1x_5 + x_1x_6 - 
	3x_1x_7 - x_2^2 - 2x_2x_3 - 4x_2x_5 + x_2x_6 - 2x_2x_7 +\\&
	2x_3^2 - x_3x_4 - 3x_3x_6 - x_3x_7 - x_4^2 - x_4x_5 - 
	x_4x_6 - x_4x_7 - x_5^2 - 3x_5x_6 + 5x_5x_7 - 2x_6^2 - 
	3x_6x_7 \\&+ 5x_7^2,\\
\end{align*}
\begin{align*}
	f_{8}&=x_1^2 + 3x_1x_2 + x_1x_3 + 4x_1x_5 + x_1x_7 - 2x_2^2 + 
	2x_2x_3 + x_2x_4 - x_2x_5 + 3x_2x_6 + 4x_2x_7 + x_3^2 + \\&
	2x_3x_4 + 4x_3x_5 + x_3x_6 - x_3x_7 + 3x_4^2 + 3x_4x_5 -
	x_4x_7 - 3x_5^2 + 2x_5x_6 - 4x_5x_7 + x_6^2 - x_6x_7\\& - 
	4x_7^2,\\
	f_{9}&=x_1^2 + x_1x_2 + x_1x_3 - 2x_1x_4 + 3x_1x_5 - 2x_1x_6 + 
	x_1x_7 - x_2x_4 + x_2x_5 - x_2x_6 + 4x_2x_7 - x_3^2 - \\&
	2x_3x_4 + 3x_3x_5 + 2x_3x_6 + 4x_3x_7 + x_4^2 - 5x_4x_6
	- x_5^2 + 5x_5x_6 + 2x_5x_7 - 3x_6^2 + 3x_6x_7 - 2x_7^2,\\
	f_{10}&=2x_1x_2 - 3x_1x_3 - 2x_1x_5 - x_1x_6 + 2x_1x_7 - 3x_2^2 + 
	x_2x_3 - 4x_2x_5 + 4x_2x_6 + 2x_2x_7 + x_3^2 +\\& 2x_3x_4 -
	x_3x_6 - 7x_3x_7 + 4x_4^2 + x_4x_5 + x_4x_6 - 2x_4x_7 - 
	x_5^2 + 2x_5x_6 + 2x_5x_7 - x_6^2 + x_6x_7,\\
	f_{11}&=x_1^2 + x_1x_2 + 3x_1x_3 + x_1x_4 + 2x_1x_5 - 2x_1x_6 - 
	5x_1x_7 + 3x_2^2 - x_2x_3 + x_2x_4 + x_2x_5 \\&- 6x_2x_6 -
	4x_2x_7 + x_3^2 - 3x_3x_4 + 3x_3x_5 - 3x_3x_6 - x_3x_7 -
	2x_4x_5 - 2x_4x_6 - x_4x_7 - 2x_5x_6\\& - 2x_5x_7 - 
	x_6x_7 - 2x_7^2,\\
	f_{12}&=x_1^2 - 2x_1x_2 + 4x_1x_3 - x_1x_4 - 2x_1x_5 + x_1x_6 + 
	x_1x_7 - x_2^2 + x_2x_3 - 4x_2x_4 + 2x_2x_5 - x_2x_6 + \\&
	4x_2x_7 + x_3^2 + 5x_3x_6 - 2x_4^2 - 2x_4x_5 - 4x_4x_6 - 
	2x_4x_7 + 2x_5^2 + x_5x_6 + 4x_5x_7 - x_6^2 + 3x_6x_7\\& + 
	4x_7^2,\\
	f_{13}&=3x_1x_2 - x_1x_3 + 2x_1x_4 + 3x_1x_5 + x_1x_6 + x_2^2 + 
	2x_2x_4 + x_2x_5 - 2x_2x_7 + x_3^2 - x_3x_4 +\\& 4x_3x_5 - 
	5x_3x_6 + 3x_3x_7 + 4x_4^2 - 5x_4x_5 + 2x_4x_6 - x_4x_7
	- x_5^2 - 6x_5x_6 - x_5x_7 - 2x_6^2 - x_7^2,\\
	f_{14}&=2x_1^2 + 3x_1x_2 - x_1x_3 + 3x_1x_7 - 2x_2^2 - 2x_2x_3 - 
	x_2x_4 + x_2x_5 + 4x_2x_6 + 3x_2x_7 - 2x_3^2\\& + 4x_3x_4 -
	4x_3x_5 + 2x_3x_6 + 4x_4^2 + x_4x_5 + 2x_4x_7 - 4x_5^2 - 
	5x_5x_6 - 4x_5x_7 - 7x_6^2 + x_6x_7.
\end{align*}
\normalsize

We need to check that the curve $C$ is everywhere locally soluble, and so represents an element of the $7$-Selmer group. Testing a genus one curve for local solubility is a well-studied problem, and we use the standard method do this. For any prime $q$ let $\widetilde{C}$ be the curve over $\mb{F}_{q}$ defined by the reduction of the above quadrics modulo $q$. By Hensel's lemma (as stated in Proposition 5, Section 2.3 of \cite{bosch2012neron}), a smooth point in $\widetilde{C}(\mb{F}_{q})$ lifts to a point $Q \in C(\mb{Q}_q)$, and so we only need to show that $\widetilde{C}$ has a smooth $\mb{F}_{q}$-point for all $q$. By  Proposition 7.4.5 of \cite{LazarThesis}, $\widetilde{C}$ is non-singular if $q$ does not divide the discriminant of $E$ or the discriminant of the field $L$. Informally, it follows from the proof of \ref{semilinear hilbert 90} that for such primes the basis $l_1,\ldots,l_p$ of $\mc{L}(D')$ reduces to a basis of $\mc{L}(\widetilde{D'})$ and so defines an isomorphism $\widetilde{E} \cong \widetilde{C}$, and hence $\widetilde{C}$ is smooth. A standard result (Lang's theorem) then implies that $\widetilde{C}(\mb{F}_{q})$ is non-empty. For the remaining primes $2,29$ and $71$, we find a smooth point in $C(\mb{F}_{q})$ by a naive search. 

Finally, to show that $C(\Q)$ is empty, we only need to check that $c_{\Q}$ is non-trivial, since $E$ is of rank zero and $E[7](\Q)$ is trivial. It suffices to show that $c_{L}=[D_{\sigma}R_{\sigma}] \in E(L)/7E(L)$ is non-trivial, i.e. that $D_{\sigma}R_{\sigma}$ is not divisible by 7, and to show this, it suffices to find a prime $\gp$ of $L$ such that the reduction $\widetilde{D_{\sigma}R_{\sigma}}$ is not divisible by $7$ in $\Tilde{E}(\mb{F}_{\gp})$. A naive search quickly shows that this is true if $\gp$ is a prime lying above 47.

We have also computed the equation of the hyperplane $H$ passing through the (images of) points $R_{\sigma^{i}}$ on $C$. The largest coefficient of the equation has 213 digits.

\end{example}
\begin{example} \label{11 example}
Finally, another interesting example we have computed is a non-trivial element of $\Sha(E_D/\Q)[11]$, where $E_D$ is the quadratic twist of the curve $37a1$ by $D=-2731$. The curve $E$ has rank 1, and this is the smallest value of $D$ for which the BSD conjecture predicts that $\Sha(E_D/\mb{Q})[11]\cong (\mb{Z}/11\mb{Z})^2$, the Heegner condition is satisfied, the curve $E_D$ is of rank 0, and the class number of $\mb{Q}(\sqrt{D})$ is equal to 11. The $11$-diagram that represents this class is a curve in $\PP^{10}$ defined by 44 quadrics, and so is impractical to print here, even though the coefficients are small integers - at most 10 in absolute value. In this rank one case, the Heegner point has much smaller height, and so the precision to which we need to compute the modular parametrization of the curve is much smaller. The class $c_L$ is given as $[D_{\sigma}R] \in E(L)/11E(L)$, where $R$ is a point of height approximately 3.812. In contrast, for the curve 8350c1, which is the smallest curve in Cremona's tables with $\Sha(E/\Q)[11]$ non-trivial, no 11-isogeny and rank zero, we have not succeeded in computing a Heegner point for any of the first few Heegner discriminants $D$ for which $p| \mathrm{Cl}(D)$.
\end{example}

\bibliographystyle{amsalpha}
\bibliography{references}
\end{document}